\documentclass{article}
\usepackage{graphicx,wrapfig}
\usepackage[english]{babel}
\usepackage{caption}
\usepackage{subcaption}
\usepackage{amsmath, amsfonts, amssymb, amsthm, bm, mathtools}
\usepackage{mathrsfs}
\usepackage{url}
\usepackage{hyperref}
\usepackage{booktabs}
\usepackage{enumitem}
\usepackage{tikz}
\usepackage{pgfplots}
\usetikzlibrary{decorations.markings, arrows.meta, calc}
\tikzset{
  arrowmid/.style={
    postaction={decorate},
    decoration={markings, mark=at position 0.55 with {\arrow{>}}}
  }
}

\theoremstyle{plain}
\newtheorem{theorem}{Theorem}[section]
\newtheorem{proposition}[theorem]{Proposition}
\newtheorem{lemma}[theorem]{Lemma}

\newtheorem{remark}[theorem]{Remark}
\newtheorem{definition}[theorem]{Definition}
\newtheorem{note}[theorem]{Nota Bene}

\theoremstyle{definition}

\newcommand{\bydef}{\stackrel{\textnormal{\tiny def}}{=}}
\newcommand{\bx}{{\bar{x}}}
\newcommand{\bu}{{\bar{u}}}
\newcommand{\bv}{{\bar{v}}}

\newcommand{\R}{{\mathbb{R}}}
\newcommand{\C}{{\mathbb{C}}}
\newcommand{\Z}{{\mathbb{Z}}}
\newcommand{\N}{{\mathbb{N}}}
\newcommand{\B}{{\mathcal{L}}}
\newcommand{\cA}{{\mathcal{A}}}
\newcommand{\cB}{{\mathcal{B}}}
\newcommand{\cF}{{\mathcal{F}}}
\newcommand{\cM}{{\mathcal{M}}}
\newcommand{\cS}{{\mathcal{S}}}
\newcommand{\cX}{{\mathcal{X}}}
\newcommand{\cY}{{\mathcal{Y}}}
\newcommand{\cZ}{{\mathcal{Z}}}
\newcommand{\cBL}{{\mathcal{BL}}}
\newcommand{\pa}{{\mathrm{pa}}}
\newcommand{\rstar}{{r_\star}}
\newcommand{\tF}{{\tilde{F}}}

\DeclareMathOperator{\DCT}{DCT}

\makeatletter
\newcommand{\owedge}{\mathbin{\mathpalette\make@circled\wedge}}
\newcommand{\make@circled}[2]{%
  \ooalign{$\m@th#1\smallbigcirc{#1}$\cr\hidewidth$\m@th#1#2$\hidewidth\cr}%
}
\newcommand{\smallbigcirc}[1]{%
  \vcenter{\hbox{\scalebox{0.77778}{$\m@th#1\bigcirc$}}}%
}
\makeatother

\usepackage{xcolor}

\begin{document}
\title{
Efficient Rigorous Continuation via Chebyshev Series Expansion I
}
\author{
Maxime Breden
\thanks
{CMAP, CNRS, \'Ecole polytechnique, Institut Polytechnique de
Paris, 91120 Palaiseau, France. {\tt maxime.breden@polytechnique.edu}.}
\and
Olivier H\'{e}not
\thanks
{National Taiwan University, Department of Mathematics, No. 1 Sec. 4 Roosevelt Rd., 10617 Taipei, Taiwan. {\tt olivierhenot@ntu.edu.tw}.}
}

\date{}

\maketitle

\begin{abstract}
We study the global continuation of solution manifolds arising in dynamical systems.
We present a rigorous continuation method based on a Chebyshev series expansion of the solution manifold.
The branch is first approximated by a high-order Chebyshev interpolation polynomial, and an explicit error bound is then obtained by verifying the contraction of a quasi-Newton operator near this approximation.
The contraction is formulated on a weighted $\ell^1$ space, giving a finer control than the typical $C^0$-error bound obtained from the uniform contraction theorem.
In fact, the latter follows directly from our contraction operator.
Furthermore, we discuss how our strategy applies naturally to pseudo-arclength continuation, where the continuation parameter fails to provide a valid local coordinate, and extends to multi-parameter continuation.
Lastly, we detail two applications in which we compute a two-parameter family of steady-states for the Cahn--Hilliard equation, and a one-parameter family of steady-states undergoing saddle-node bifurcations for the Shigesada--Kawasaki--Teramoto system.
\end{abstract}

\begin{center}
\end{center}




\section{Introduction}\label{sec:introduction}

The study of dynamical systems, whether described by ordinary differential equations (ODEs), partial differential equations (PDEs), delay differential equations (DDEs), or discrete maps, is a matter of invariant sets (e.g., fixed-points, periodic orbits, and invariant manifolds).
They constitute organizing centers for the global dynamics, and their bifurcation, emergence, or disappearance, illuminates the mechanisms behind rich and complex behaviours.
Yet, finding such objects is already a nontrivial task, let alone understanding their dependence on parameter variations.
The theory of bifurcation is by now well established~\cite{Kuz13,GucHol13}, but rigorously applying it to a given nonlinear system can still be difficult, and typically only provides local information.
On the other hand, many software libraries have been developed for numerical continuation and bifurcation analysis, e.g., for ODEs~\cite{AUTO07,MATCONT08,COCO13}, PDEs~\cite{pde2path} and DDEs~\cite{DDEBIFTOOL02}.
These tools excel at sketching out bifurcation diagrams, and displaying large branches of invariant sets.
However, they still fall short of proving (in the rigorous mathematical sense) the existence of the branches and bifurcations.

In the past thirty years, mathematical methods have been designed to exploit numerical techniques and \emph{validate a posteriori} global families depending on parameters~\cite{AriKoc10,BerLesMis10,BreLesVan13,GamLesPug16,Plu95,Wan18}.
Proceeding from an approximate branch of solutions, obtained for instance by means of numerical continuation, these rigorous continuation methods establish the existence of a genuine family of solutions in a small and explicit neighborhood of the approximation.
Rigorous continuations have already been used successfully in many contexts, including for instance uniqueness results for a family of semilinear elliptic equations~\cite{McKPacPluRot09,McKPacPluRot12}, an exhaustive description of a bifurcation diagram~\cite{AriKoc10}, the rigorous computation of a bifurcation diagram delimiting energy minimizers for diblock copolymers~\cite{BerWil17}, the resolution of the Marchal conjecture in the three-body problem~\cite{CalGarHenLesMir24}, or the progresses toward a conjecture about the existence of slowly oscillating periodic solutions in a DDE~\cite{Les10} and another conjecture concerning existence of traveling waves in a fourth-order PDE~\cite{BerBreLesMur18}.

More broadly, a rigorous continuation approach is an extension of a class of computer-assisted proofs at fixed parameters, which we now briefly describe.
Consider a zero-finding problem $F(x) = 0$, and an approximate zero $\bx$ of $F$, typically obtained numerically.
The objective is to prove a posteriori that a zero of $F$ exists near $\bx$, by verifying that a Newton-like operator associated to $F$ is a contraction in the vicinity of $\bx$.
An example of such statement is given below.

\begin{theorem}\label{th:NK}
Let $\cX,\cY$ be Banach spaces, $\rstar \in (0, \infty]$, $\bx \in \cX$, $F : \cX \to \cY$ a $C^1$ map, and $A:\cY\to\cX$ an injective linear map.
Assume that there exist positive constants $Y$, $Z_1$, $Z_2 = Z_2(\rstar)$ satisfying
\begin{subequations}
\label{eq:bounds}
\begin{align}
\left\| AF(\bx) \right\|_{\cX} &\leq Y \label{eq:YNK}\\
\left\| I - ADF(\bx) \right\|_{\B(\cX,\cX)} &\leq Z_1 \label{eq:Z1NK}\\
\left\| A\left(DF(x) - DF(\bx)\right) \right\|_{\B(\cX,\cX)} &\leq Z_2 \left\| x-\bx\right\|_{\cX}, \qquad \text{for all } x \in B_{\rstar}(\bx), \label{eq:Z2NK}
\end{align}
\end{subequations}
where $B_{\rstar}(\bx)$ denotes the closed ball centered $\bx$ with radius $\rstar$ in $\cX$, and $\B(\cX,\cX)$ is the space of bounded linear operators from $\cX$ to $\cX$.
If there exists $r \in (0,\rstar]$ such that
\begin{subequations}\label{eq:condr}
\begin{align}
Y + r Z_1 + \frac{r^2}{2}Z_2  &< r \label{eq:condr1NK}\\
Z_1 + r Z_2 &< 1 \label{eq:condr2NK},
\end{align}
\end{subequations}
then the map $x\mapsto x - AF(x)$ is a contraction in $B_{\rstar}(\bx)$, and therefore $F$ has a unique zero $x^\star$ in $B_r(\bx)$.
\end{theorem}

So as to obtain a contraction, a natural choice is to take $A$ equal to $DF(\bx)^{-1}$, or a suitable approximation of $DF(\bx)^{-1}$ whenever the exact inverse is computationally or analytically demanding.
The application to dynamical systems of statements similar to Theorem~\ref{th:NK} can be found in a vast body of work, see, e.g.,~\cite{BerLes15,NakPluWat19}.
For a proof of the specific version used here, we refer to~\cite[Theorem 1.2.2]{Bre25}.

Based on the above discussion, let us suppose that the problem depends on a parameter $\lambda$; specifically, we consider a continuation problem of the form
\begin{equation}\label{eq:Fxl}
F(x, \lambda) = 0,
\end{equation}
where $F$ is a map from $\cX\times \Lambda$ to $\cY$, and $\Lambda$ is a compact subset of $\R$.
Assume that for some fixed $\lambda = \lambda_0$, a suitable approximate solution $\bar{x} = \bar{x}_0$ and operator $A = A_0$ have been found, together with estimates $Y, Z_1, Z_2$ satisfying~\eqref{eq:bounds} and~\eqref{eq:condr}, so that Theorem~\ref{th:NK} guarantees the existence of a true zero $x_0^\star$ of $F(\cdot,\lambda_0)$ near $\bx_0$.
Provided that the map $F$ is continuous with respect to $\lambda$, we can expect the obtained bounds $Y, Z_1, Z_2$ to also depend continuously on  $\lambda$.
Hence, the open conditions~\eqref{eq:condr} ought to hold for all $\lambda$ in some neighborhood of $\lambda_0$.
In other words, the parameter dependent fixed-point map $x \mapsto x - A_0 F(x, \lambda)$ is a \emph{local uniform contraction in $\lambda$}, thereby implying the existence of a family $\lambda \mapsto x^\star(\lambda)$ of zeros of $F(\cdot, \lambda)$ close to the constant approximation $\lambda \mapsto \bar{x}_0$.
This approach is very crude and the neighborhood of $\lambda_0$ on which \eqref{eq:condr} holds will likely be very small; after all, it is a zeroth-order approximation of the family.
To date, the usual course in the literature on rigorous continuation is to employ the first-order approximation $\bx(\lambda) = \frac{\lambda_1 - \lambda}{\lambda_1 - \lambda_0} \bar{x}_0 + \frac{\lambda - \lambda_0}{\lambda_1 - \lambda_0} \bar{x}_1$, given two nearby parameters $\lambda_0$ and $\lambda_1$, and two corresponding approximate solutions $\bar{x}_0$ and $\bar{x}_1$.
In most cases, the approximate inverse $A$ is still taken constant and equal to $A_0$ for all $\lambda$ in $[\lambda_0,\lambda_1]$, but linear interpolation for $A$ has also been used~\cite{BerQue21}.
Then, the dependency with respect to $\lambda$ is tracked in all the estimates $Y, Z_1, Z_2$ to derive uniform bounds over $[\lambda_0, \lambda_1]$.
Usually this still requires a relatively small segment $[\lambda_0, \lambda_1]$ to satisfy the uniform contraction conditions, necessitating to repeat the process over multiple intervals, with additional verification to ensure that each piece connects to the next to form the branch.
In the end, the global solution branch is obtained using a piecewise linear approximation and a combination of local proofs.
This approach extends to two-parameters families, but the modifications are rather involved with simplicial decompositions, and in practice the validation procedure is computationally expensive~\cite{GamLesPug16,ChuQue23}.

In contrast, an alternative approach has recently been successful in using a global, high-order approximation of a solution curve, based on Taylor expansions in~\cite{AriGazKoc21}, and on more general series expansions in~\cite{Bre23}.
This new method allows for several significant improvements over the existing low-order and local rigorous continuation techniques.
We present and detail in this paper the use of Chebyshev series expansions with respect to parameters $\lambda$, taking values in a compact subset $\Lambda$ of $\R^d$.
We highlight four notable features of our strategy:
\begin{enumerate}
\item \textbf{Analytical estimates.}
The seemingly natural space in which to derive the continuation estimates is $C^0(\Lambda,\cX)$, the space of continuous functions over $\Lambda$ with values in $\cX$, as we want to apply the uniform contraction mapping theorem on $\cX$.
However, the high-order approximations based on series expansions suggest to instead derive the continuation estimates in $\ell^1_\omega(\cX)$, a subspace of $\cX^\mathbb{N}$ with a weighted $\ell^1$-norm (defined in Section~\ref{sec:parameter_continuation}).
A particularly attractive feature of this setting is that the continuation estimates become straightforward generalizations of the pointwise (i.e., at fixed parameter values) estimates.
This was already noticed on some examples in~\cite{Bre23}, and will be explored more generally in this paper.

\item \textbf{Computational cost.}
For pointwise computer-assisted proofs of the type discussed above, a crucial role is played by a finite-dimensional subspace, say of dimension $K$, used to represent the approximate solution $\bx$ and construct the approximate inverse $A$, assuming for simplicity that the same subspace is used for both.
The proof entails manipulating $K$-by-$K$ matrices, and its computational bottleneck is related to how large $K$ must be.
When employing a series expansion in the $\lambda$ variable, yet another finite-dimensional subspace is introduced, say of dimension $N$, and combining the two leads in general to objects of size $K^2 N^2$, see~\cite{Bre23}.
In this work, we perform continuation using Chebyshev expansions due to their excellent approximation properties~\cite{Tre13} and their close connection to Fourier series.
This relationship leads to efficient computation of operations such as multiplication, evaluation, and interpolation via the Discrete Fourier Transform, for which Fast Fourier Transform (FFT) algorithms can be leveraged.
By exploiting these properties, we develop rigorous continuations algorithms that operate on data structures of size at most $K^2 N \log N$.
Their implementation is available via the open-source software \texttt{RadiiPolynomial} \cite{RadiiPolynomial.jl}.

\item \textbf{Analytic parameterization.}
If $F$is sufficiently regular with respect to $\lambda$, then one can in principle use the implicit function theorem to control derivatives of $\lambda \mapsto x^\star(\lambda)$.
However, obtaining explicit estimates on derivatives of $x^\star$ from a $C^0$ control provided by rigorous continuation requires extra work~\cite{AriKoc10}.
In contrast, in this article, we prove the contraction in a suitably weighted $\ell^1$ space, modeling the regularity of the solution.
This leads to a stronger contraction argument than the $C^0$-norm, and directly provides us with $C^k$ or analytic estimates on $\lambda \mapsto x^\star(\lambda)$.

\item \textbf{Multi-dimensional continuation.}
Last but not least, rigorous continuation in weighted $\ell^1$ spaces generalizes easily to multi-parameter continuation ($d \ge 2$), yielding an approach which is both simpler and more efficient than the available methods thus far in the literature; see, e.g., \cite{GamLesPug16} for the two-parameters setting.
The latter reference might seem more flexible at first glance, as it is based on a fine triangulation of the solution manifold together, and can therefore accommodate to parameter spaces of arbitrary shapes.
Although multivariate Chebyshev series are naturally defined on hyper-cuboidal regions, we will showcase that one can in fact embed those into arbitrary shapes without difficulty.
Nevertheless, triangulation remains essential for multi-parameter problems, as it provides a practical mean of sampling and organizing values on a solution manifold with intricate geometry.
An illustration of such complicated manifold can be found in the planar circular restricted four-body problem (PCR4BP), where numerical simulations of the manifold of equilibria hint on the occurrence of a triple-cusp bifurcation, knotting the manifold. 
This example, as well as techniques dedicated to optimizing two-dimensional parameter continuation are the subject of a follow-up work.
\end{enumerate} 

The article is organized as follows.
In Section~\ref{sec:parameter_continuation}, we lay out the functional analytic framework for rigorous continuation using Chebyshev expansions in $\ell^1_\omega(\cX)$, and explain how to take advantage of the FFT in practice.
We then discuss in Section~\ref{sec:reparam} the choice of parametrization of the family, first for one-parameter continuation using the curve's arclength, second for multi-parameter continuation including non-rectangular parameter regions.
Lastly, in Section~\ref{sec:applications}, we apply our method to prove a two-parameter family of steady-states for the Cahn--Hilliard equation, and a one-parameter family of steady-states, undergoing saddle-node bifurcations, for the Shigesada--Kawasaki--Teramoto system.



\section{Rigorous parameter continuation}\label{sec:parameter_continuation}

To simplify the exposition in this section, we restrict our attention to a single parameter, and, without loss of generality, we take the parameter domain to be $\Lambda = [-1, 1]$.
Moreover, we consider the case where the solution curve can be expressed as a function of the continuation parameter $\lambda$.
The multi-parameter continuation, as well as situations in which the implicit function theorem does not hold, are deferred to Section~\ref{sec:reparam}.

\subsection{Chebyshev interpolation, series and weighted $\ell^1$ spaces}
\label{sec:basic_stuff_Cheb}

This section reviews background material and introduces the notation followed throughout the article.

Our strategy to prove the existence of a family of solutions $x^\star : [-1,1] \to \cX$, where $\cX$ is some Banach space, begins with its approximation.
It is well known that Lagrange interpolation at Chebyshev nodes is an efficient and practical method for approximating regular functions~\cite{Tre13}.
Given an integer $N \geq 1$, we consider the Chebyshev nodes (of the second kind) given by
\begin{equation}\label{eq:chebyshev_nodes}
\lambda^{(j)} = -\cos \Big( \frac{j\pi}{N} \Big), \qquad j = 0, \dots, N.
\end{equation}
Then, for a set of $N+1$ points $\{x^{(j)}\}_{j=0}^N$ in $\cX^{N+1}$ associated to the above nodes, its interpolation polynomial is given by
\begin{equation}\label{eq:poly}
\lambda \mapsto \sum_{n=-N}^N x_{| n|} T_{| n|}(\lambda),
\end{equation}
where $T_n$ is the $n$-th Chebyshev polynomial of the first kind, recursively given by
\begin{equation}
T_0(\lambda) = 1, \qquad T_1(\lambda) = \lambda, \qquad T_n (\lambda) = 2\lambda T_{n-1}(\lambda) - T_{n-2}(\lambda), \quad n\ge 2,
\end{equation}
and which satisfies the insightful identity
\begin{equation}\label{eq:cheb_identity}
T_n(\cos(\theta)) = \cos(n\theta).
\end{equation}
In particular, the Chebyshev coefficients $y_n$ of the interpolation polynomial can be obtained using the discrete cosine transform (DCT).
More precisely, given $\{x_n\}_{n=0}^N$ and $\{x^{(j)}\}_{j=0}^N$, the DCT and its inverse transform $\DCT^{-1}$ are given by
\begin{subequations}\label{eq:DCT}
\begin{align}
\left[\DCT\left(\{x^{(j)}\}_{j=0}^N\right)\right]_n &\bydef
\begin{cases}
\displaystyle\frac{1}{2N}\left( x^{(N)} + (-1)^n \sum_{j = -N+1}^{N-1} x^{(|j|)} \cos\left( j \frac{n \pi}{N} \right) \right), & n=0,\dots,N-1, \\
\displaystyle\frac{1}{4N}\left( x^{(N)} + (-1)^N \sum_{j = -N+1}^{N-1} (-1)^j x^{(|j|)} \right), & n = N, \\
\end{cases} \\
\left[\DCT^{-1}\left(\{x_n\}_{n=0}^N\right)\right]^{(j)} &\bydef \sum_{n = -N}^N (-1)^n x_{|n|} \cos \Big(n \frac{j\pi}{N} \Big), \qquad j = 0, \dots, N, \\
\end{align}
\end{subequations}
Thus, the Chebyshev coefficients of the polynomial~\eqref{eq:poly} are given by $x_n = \left[\DCT\left(\{x^{(j)}\}_{j=0}^N\right)\right]_n$ for $n = 0, \ldots, N$, and its values at the Chebyshev nodes~\eqref{eq:chebyshev_nodes} are given by $x^{(j)} = \left[\DCT^{-1}\left(\{x_n\}_{n=0}^N\right)\right]^{(j)}$ for $j = 0, \dots, N$.
Notably, the DCT and its inverse can be computed efficiently using FFT algorithms having complexity $O(N\log N)$.

As much as we use polynomials in the Chebyshev basis to approximate the solution curve, we exploit Chebyshev series to represent and prove its existence.
Recall that any Lipschitz continuous function $\psi : [-1,1] \to \cX$ admits a unique, absolutely converging, Chebyshev series
\[
\psi(\lambda) = \sum_{n \in \Z} \psi_{|n|} T_{|n|} (\lambda).
\]
The decay of the coefficients $\psi_n$ reflects the regularity of $\psi$; in the case of an analytic function, the coefficients decay to zero at least geometrically.
Chebyshev series therefore have clear advantages over Taylor series: $(i)$ they can represent functions that are merely Lipschitz, and $(ii)$ they converge rapidly for any analytic function on $[-1,1]$ even when its Taylor series diverges due to the presence of poles in the complex plane.

We model this decay rate via the weighted sequence space
\begin{equation}\label{eq:weighted_ell1}
\ell^1_\omega(\cX) \bydef \left\{ \psi \in \cX^{\mathbb{N}} \, : \, \| \psi \|_{\ell^1_\omega(\cX)} \bydef \sum_{n \in \mathbb{Z}} \left\| \psi_{|n|} \right\|_\cX \omega_{| n|} < \infty \right\},
\end{equation}
where $\omega \in (0, \infty)^\N$ is a sequence of weights satisfying the assumptions:
\begin{enumerate}
\renewcommand{\theenumi}{\textbf{(A\arabic{enumi})}}
\renewcommand{\labelenumi}{\theenumi}
\item\label{A1} $\omega_n\ge 1$ for all $n\ge 1$,
\item\label{A2} $\omega_{| k+n|} \le \omega_{| k|}\omega_{| n|}$ for all $k$ and $n$ in $\Z$.
\end{enumerate}
Two standard examples of such weight sequences are
\begin{subequations}
\begin{alignat}{4}
\omega_n &= \mu^n,\quad &&\text{for all } n \in \N, \qquad &&\text{for some }\mu \ge 1, \qquad &&\text{(Geometric weight)} \label{eq:analytic_weights} \\
\omega_n &= (1+n)^\alpha,\quad &&\text{for all }n\in\N, \qquad &&\text{for some }\alpha>0. \qquad &&\text{(Algebraic weight)} \label{eq:Ck_weights}
\end{alignat}
\end{subequations}
Choosing~\eqref{eq:analytic_weights} with $\mu > 1$ guarantees that the map $\psi \in \ell^1_\omega(\cX)$ extends analytically in the Bernstein ellipse 
$\{ z \in \C \, : \, z = \frac{1}{2} (w + w^{-1}), \, 1 \le |w| < \mu \}$.
On the other hand, the choice~\eqref{eq:Ck_weights} corresponds to functions admitting a finite number of derivatives.
We refer to~\cite[Chapter 7]{Tre13} for sharp estimates relating the regularity of a function and the decay of its Chebyshev coefficients.
Although larger values of $\mu$ in \eqref{eq:analytic_weights} (resp. $\alpha$ in \eqref{eq:Ck_weights}) yield stronger analyticity properties (resp. higher regularity), they bear a cost when verifying the contraction conditions \eqref{eq:condr} as they lead to larger norm estimates in \eqref{eq:bounds}.

\begin{remark}\label{rem:notation}
Throughout this paper, we identify a sequence $\psi = (\psi_0, \psi_1, \dots)$ of Chebyshev coefficients in $\ell^1_\omega(\cX)$ with the corresponding function $\psi:\lambda \mapsto \sum_{n \in \mathbb{Z}} \psi_{|n|} T_{|n|}(\lambda)$.
\end{remark}

The key principle of our rigorous continuation method is to recast the continuation problem~\eqref{eq:Fxl} as a contraction on such weighted sequence spaces.
As a result, the error between the approximate solution curve $\bx$ and the exact one $x^\star$ gives readily a quantitative control on the $C^0$-norm and the derivatives.
Indeed, from assumption~\ref{A1}, for any $\psi\in \ell^1_\omega(\cX)$, the corresponding function $\psi : [-1,1]\to \cX$ is at least continuous and we have
\begin{equation}\label{eq:C0vsell1}
\sup_{\lambda\in[-1,1]} |\psi(\lambda)| \le \| \psi \|_{\ell^1_\omega(\cX)}.
\end{equation}
Additionally, using the Cauchy integral formula (see also~\cite{BleBluBreEng25} for sharper bounds), we have the following estimates, which provide control on derivatives of $\psi$.

\begin{lemma}
\label{lem:derivatives_analytic}
Suppose that the weights are of the form~\eqref{eq:analytic_weights}, with $\mu>1$. Then, for all $k \ge 1$ and all $\psi \in \ell^1_\omega(\cX)$,
\begin{equation}
\sup_{\lambda\in[-1,1]} \left| \psi^{(k)}(\lambda)\right| \le \frac{k!}{\left(\frac{\mu+\mu^{-1}}{2}-1\right)^k} \left\| \psi\right\|_{\ell^1_\omega(\cX)}.
\end{equation}
\end{lemma}

\noindent
Alternatively, for algebraic weights, the next result holds.

\begin{lemma}\label{lem:derivatives_Ck}
Suppose that the weights are of the form~\eqref{eq:Ck_weights}.
\begin{itemize}
\item If $\alpha=2$, then, for all $\psi\in\ell^1_\omega(\cX)$,
\begin{equation}
\sup_{\lambda\in[-1,1]} \left| \frac{\mathrm{d}}{\mathrm{d}\lambda} \psi(\lambda)\right| \le \| \psi \|_{\ell^1_\omega(\cX)}.
\end{equation}
\item If $\alpha=4$, then, for all $\psi\in\ell^1_\omega(\cX)$
\begin{equation}
\sup_{\lambda\in[-1,1]} \left| \frac{\mathrm{d}^2}{\mathrm{d}\lambda^2} \psi(\lambda)\right| \le \frac{1}{3}\| \psi\|_{\ell^1_\omega(\cX)}.
\end{equation}
\end{itemize}
\end{lemma}
\begin{proof}
Starting from the derivative formula for Chebyshev series,
\[
\frac{\mathrm{d}}{\mathrm{d}\lambda} \psi(\lambda) = 2 \sum_{l \ge 0} \left((2l+1)u_{2l+1} + 2\sum_{k \ge 1} (k+2l+1)u_{k+2l+1}T_k(\lambda)\right),
\]
we obtain
\[
\left|\frac{\mathrm{d}}{\mathrm{d}\lambda} \psi(\lambda)\right| \le  2 \sum_{m \ge 1} m^2 |u_m| \le \sum_{m \in \Z}  |u_m| (1 + |m|)^2.
\]
A similar calculation yields
\[
\left|\frac{\mathrm{d}^2}{\mathrm{d}\lambda^2} \psi(\lambda)\right| \le  \frac{2}{3} \sum_{m \ge 2} m^2(m^2-1) |u_m| \le \frac{1}{3}\sum_{m\in\Z}  |u_m| (1+| m|)^4. \qedhere
\]
\end{proof}

We conclude this section by introducing a discrete convolution product, in a slightly broader setting than usual.
Consider Banach spaces $\cX$, $\cY$ and $\cZ$, equipped with a bounded bilinear map $\odot$ such that, for all $x\in\cX$ and $y\in\cY$,
\begin{subequations}\label{eq:opX}
\begin{align}
&\odot:\cX\times\cY \to \cZ, \\
&\| x\odot y \|_\cZ \le \| x \|_\cX \| y \|_\cY.
\end{align}
\end{subequations}
We call \emph{discrete convolution product} the map $\ast:\ell^1_\omega(\cX)\times\ell^1_\omega(\cY) \to \ell^1_\omega(\cZ)$ given, for all $\phi\in\ell^1_\omega(\cX)$ and $\psi\in\ell^1_\omega(\cY)$, component-wise by
\begin{equation}\label{eq:conv}
(\phi * \psi)_n \bydef \sum_{n' \in \Z} \phi_{|n - n'|} \odot \psi_{|n'|}, \qquad n\in\N.
\end{equation}
The fact that this discrete convolution product is well-defined follows readily from assumption~\ref{A2}, and we have
\begin{equation}\label{eq:banach_algebra}
\| \phi * \psi \|_{\ell^1_\omega(\cZ)} \le \| \phi \|_{\ell^1_\omega(\cX)} \| \psi \|_{\ell^1_\omega(\cY)}, \qquad \text{for all }(\phi,\psi)\in\ell^1_\omega(\cX)\times\ell^1_\omega(\cY).
\end{equation}
This discrete convolution product is the natural product rule associated to $\odot$, in the sense that
\begin{equation*}\label{eq:prodlambda}
\phi(\lambda) \odot \psi(\lambda) = \left[\phi * \psi\right](\lambda), \qquad \text{for all } \lambda \in[-1,1].  
\end{equation*}
Moreover, for any $\phi \in \ell^1_\omega(\cX)$, we also get a multiplication operator $\cM(\phi) \in\B(\ell^1_\omega(\cY),\ell^1_\omega(\cZ))$ defined by $\cM(\phi)\psi \bydef \phi\ast \psi$ for all $\psi \in\ell^1_\omega(\cY)$, and which satisfies
\begin{equation}
\| \cM(\phi) \|_{\B(\ell^1_\omega(\cY),\ell^1_\omega(\cZ))} \le \| \phi \|_{\ell^1_\omega(\cX)}.
\end{equation}
One classical setting is the one where $\cX=\cY=\cZ$ is a Banach algebra with multiplication $\odot$. Then, provided the Banach algebra is unital, the norm of the multiplication operator is actually equal to the norm of the underlying element:
\begin{equation}\label{eq:opnormmult}
\| \cM(\phi) \|_{\B(\ell^1_\omega(\cX),\ell^1_\omega(\cX))} = \| \phi \|_{\ell^1_\omega(\cX)}.
\end{equation}
Another important situation for this paper is when 
$\odot:\B(\cX,\cY)\times \cX \to \cY$ denotes the application of a linear map to a vector, 
in which case any element $A \in \ell^1_\omega(\B(\cX,\cY))$ 
gives rise to a multiplication operator $\cM(A)\in\B(\ell^1_\omega(\cX),\ell^1_\omega(\cY))$.

With the main tools in place, we begin by illustrating our rigorous continuation method using a simple example in the next subsection.
A more general treatment for possibly infinite dimensional Banach spaces $\cX$ will be given afterwards in Section~\ref{sec:framework}.

\subsection{An illustrative example: the cubic root}\label{sec:cubic_root}

We consider the problem of obtaining a rigorous enclosure of the real solution curve to
\begin{equation}\label{eq:cubic}
0 = F(x, \lambda) = x^3 - \lambda - 2, \qquad \text{for all } \lambda \in [-1, 1].
\end{equation}
We first show how Theorem~\ref{th:NK} can be applied to obtain a solution for a fixed value of $\lambda$, and then present two different ways of obtaining the entire solution curve on $[-1, 1]$. Obviously this is a very basic example, which is only here for pedagogical purposes. More involved applications are presented in Section~\ref{sec:applications}.

\subsubsection{Proof of a solution at a fixed parameter}

Fix a parameter value $\lambda_\circ \in [-1, 1]$.
In view of applying Theorem~\ref{th:NK} to $F(\cdot,\lambda_\circ)$, here with $\cX = \cY = \R$ and $\| \cdot \|_\cX = | \cdot |$, suppose that an approximate zero $\bx_\circ$ of $F(\cdot,\lambda_\circ)$ and approximate inverse of $D_x F(\bx_\circ, \lambda_\circ)$ have been found:
\begin{equation*}
\bx_\circ \approx \sqrt[^3]{\lambda_\circ + 2}, \qquad
A_\circ \approx D_x F(\bx_\circ, \lambda_\circ)^{-1} = \frac{1}{3\bx_\circ^2}.
\end{equation*}
The $\approx$ symbol highlights that none of the operations involved in the right-hand side need to be rigorous, and can be performed using floating-point arithmetic.

In this simple case, obtaining the required estimates~\eqref{eq:bounds} is straightforward.
Indeed,
\begin{align*}
| A_\circ F(\bx_\circ, \lambda_\circ) | &= | A_\circ (\bx_\circ^3 - \lambda_\circ - 2) | , \\
| 1 - A_\circ D_x F(\bx_\circ, \lambda_\circ) | &= | 1 - 3 A_\circ \bx_\circ^2 |,
\end{align*} 
and, for any $\rstar>0$ and all $x\in B_{\rstar}(\bx_\circ)$,
\begin{align*}
| A_\circ (D_x F(x, \lambda_\circ) - D_x F(\bx_\circ, \lambda_\circ)) |  
&\le |A_\circ| \left(\sup_{x\in B_{\rstar}(\bx_\circ)} | D_x^2F(x,\lambda_\circ)|\right) | x - \bx_\circ | \\
&\le  6|A_\circ| (| \bx_\circ| + \rstar) | x - \bx_\circ |.
\end{align*}
Therefore, we can take
\begin{subequations}\label{eq:cubic_bounds_circ}
\begin{align}
Y &= | A_\circ (\bx_\circ^3 - \lambda_\circ - 2) |, \\
Z_1 &= | 1 - 3 A_\circ \bx_\circ^2 | , \\
Z_2 &= 6|A_\circ| (| \bx_\circ| + \rstar),
\end{align}
\end{subequations}
which can be rigorously computed using interval arithmetic (see, e.g., the software library \cite{IAjl}).
Assuming that the two inequalities~\eqref{eq:condr} hold, Theorem~\ref{th:NK} yields the existence of a zero $x^\star_\circ$ of $F(\cdot,\lambda_\circ)$ near $\bx_\circ$:
\begin{equation}
|x^\star_\circ - \bx_\circ| \le r.
\end{equation}

\begin{note}[Injectivity of $A$]
In order to be allowed to use Theorem~\ref{th:NK}, we should make sure that $A_\circ$ is invertible. This is straightforward, as we simply need to check that the $A_\circ$ we select is non-zero, but we note that this verification is done automatically. Indeed, inequality~\eqref{eq:condr2NK} implies $Z_1<1$, which already enforces $A_\circ\neq 0$. Even though the argument is completely trivial here and barely worth mentioning, we do emphasize it because getting the injectivity of $A$ from the fact that $Z_1<1$ is a common argument that will also be used in more complex situations later on.
\end{note}

\subsubsection{Rigorous continuation in $C^0([-1,1],\cX)$} 
\label{sec:sqrt_C0}

As mentioned in the introduction, one can use a uniform version of Theorem~\ref{th:NK}, such as the following statement, to validate the solution curve $\lambda \mapsto x^\star(\lambda)$ to~\eqref{eq:cubic}.

\begin{theorem}\label{th:NKpara}
Let $\cX,\cY$ be Banach spaces, $\Lambda$ a simply connected compact subset of $\R^d$, $\rstar \in (0, \infty]$, $\bx:\Lambda\to\cX$ a $C^0$ function, $F : \cX \times \Lambda \to \cY$ a $C^0$ map, continuously differentiable with respect to the first variable, and $A:\Lambda \to \B(\cY,\cX)$ a $C^0$ map such that $A(\lambda)$ is injective for all $\lambda\in\Lambda$.
Assume that there exist positive constants $Y$, $Z_1$, $Z_2 = Z_2(\rstar)$ satisfying, for all $\lambda\in\Lambda$,
\begin{subequations}\label{eq:bounds_cont}
\begin{align}
\left\| A(\lambda)F(\bx(\lambda),\lambda) \right\|_{\cX} &\le Y, \label{eq:YNK_cont}\\
\left\| I - A(\lambda)D_xF(\bx(\lambda),\lambda) \right\|_{\B(\cX,\cX)} &\le Z_1, \label{eq:Z1NK_cont}\\
\left\| A(\lambda)\left(D_xF(x,\lambda) - D_xF(\bx(\lambda),\lambda)\right) \right\|_{\B(\cX,\cX)} &\le Z_2 \left\| x-\bx(\lambda)\right\|_{\cX}, \qquad \text{for all } x \in B_{\rstar}(\bx(\lambda)) . \label{eq:Z2NK_cont}
\end{align}
\end{subequations}
If there exists $r \in (0,\rstar]$ such that the conditions~\eqref{eq:condr} are met for the bounds $Y, Z_1, Z_2$ herein, then the map $(h,\lambda) \mapsto h-A(\lambda)F(\bx(\lambda)+h,\lambda)$ is a uniform contraction on $B_r(0)\times\Lambda$.
Therefore, for all $\lambda\in\Lambda$, the map $F(\cdot,\lambda)$ has a unique zero $x^\star(\lambda)$ in $B_r(\bx(\lambda))$ and the function $\lambda \mapsto x^\star(\lambda)$ is $C^0$.
\end{theorem}

Let us verify the hypotheses of Theorem~\ref{th:NKpara} to $F$ from~\eqref{eq:cubic}, here with $\Lambda=[-1,1]$, $\cX=\cY=\R$ and $\|\cdot\|_\cX = |\cdot|$.
We need to construct suitable families $\lambda \mapsto \bx(\lambda)$ and $\lambda \mapsto A(\lambda)$, and derive the bounds $Y, Z_1, Z_2$ satisfying~\eqref{eq:bounds_cont}.

First, fix some integer $N \geq 1$ and consider the corresponding Chebyshev nodes $\lambda^{(j)}$ defined in~\eqref{eq:chebyshev_nodes}.
For each node, compute an approximate zero $\bx^{(j)}$ of $F(\cdot,\lambda^{(j)})$ and an approximate inverse $A^{(j)}$ of $D_x F(\bx^{(j)},\lambda^{(j)}) = 3 (\bx^{(j)})^2$.
The functions $\lambda \mapsto \bx(\lambda)$ and $\lambda \mapsto A(\lambda)$ are constructed by computing the corresponding interpolation polynomials, i.e., $\bx \approx \DCT\left(\{\bx^{(j)}\}_{j=0}^N\right)$ and $A \approx \DCT\left(\{A^{(j)}\}_{j=0}^N\right)$.
Once again, we use the $\approx$ symbol to emphasize that, for our purposes, the DCT may be computed approximately, e.g., using floating-point arithmetic; see also Remark~\ref{rem:approx_interp} about when performing a rigorous DCT at this stage is essential.

Then, for any $\rstar > 0$ and $\lambda\in[-1,1]$, we obtain
\begin{align*}
| A(\lambda) F(\bx(\lambda), \lambda) | &= | A(\lambda) (\bx(\lambda)^3 - \lambda - 2) |, \\
| 1 - A(\lambda) D_x F(\bx(\lambda), \lambda) | &= | 1 - 3 A(\lambda) \bx(\lambda)^2 | , \\
| A(\lambda) (D_x F(x, \lambda) - D_x F(\bx(\lambda), \lambda)) | &\le 6 |A(\lambda)| (| \bx(\lambda)| + \rstar) | x - \bx(\lambda) |, \qquad \text{for all }  x\in B_{\rstar}(\bx(\lambda)).
\end{align*}
The right-hand-sides of the above inequalities can be bounded by estimating the supremum over $\lambda$ in $[-1,1]$, which could be directly computed using interval arithmetic.
Yet, mitigating error propagation, it is easier to use the upper bound provided by the $\ell^1_\omega$-norm~\eqref{eq:C0vsell1}, so that we define
\begin{subequations}\label{eq:cubic_bounds_C0}
\begin{align}
Y &= \left\| A * (\bx^{*3} - e_1 - 2e_0) \right\|_{\ell^1_\omega(\R)}, \\
Z_1 &= \left\| e_0-3A * \bx^{*2} \right\|_{\ell^1_\omega(\R)} , \\
Z_2 &= 6\left\| A  \right\|_{\ell^1_\omega(\R)} \left(\left\| \bx  \right\|_{\ell^1_\omega(\R)} + \rstar \right),
\end{align}
\end{subequations}
where $e_0, e_1 \in \ell^1_\omega(\R)$ stand for the unit-norm basis elements
\begin{equation}\label{eq:identity}
e_0 \bydef (1, 0, 0, 0 \dots), \qquad
e_1 \bydef (0, 1/2, 0, 0, \dots),
\end{equation}
or, equivalently, $e_0$ is the constant function $e_0(\lambda) = 1$ and $e_1$ the identity function $e_1(\lambda) = \lambda$.

For any weight sequence $\omega$ for which the assumptions~\ref{A1}-\ref{A2} hold, the constants $Y$, $Z_1$ and $Z_2$ from~\eqref{eq:cubic_bounds_C0} satisfy~\eqref{eq:bounds_cont}.
Assuming that the inequalities~\eqref{eq:condr} also hold, Theorem~\ref{th:NKpara} yields the existence of a solution curve $\lambda \mapsto x^\star(\lambda)$ near $\bx$:
\begin{equation}\label{eq:err_C0}
\left| x^\star(\lambda) - \bx(\lambda) \right| \le r, \qquad \text{for all } \lambda\in[-1,1].
\end{equation}
\begin{note}[Injectivity of $A(\lambda)$]
Here again, we note that the injectivity of $A(\lambda)$ for each $\lambda\in[-1,1]$ is a consequence of $Z_1<1$ (implied by the inequality~\eqref{eq:condr2NK}), that is
\begin{align*}
\left\vert 1 -3A(\lambda)\bx(\lambda)^2 \right\vert \leq \left\| e_0-3A * \bx^{*2} \right\|_{\ell^1_\omega(\R)} < 1 \qquad \text{for all }\lambda\in[-1,1].
\end{align*}
\end{note}

\subsubsection{Rigorous continuation in $\ell^1_\omega(\cX)$}
\label{sec:sqrt_ell1}

While the rigorous continuation approach proposed in Section~\ref{sec:sqrt_C0} is perfectly suitable, and arguably very natural, we present here a slightly different viewpoint, which turns out to be more practical in more sophisticated problems.
We go back to using Theorem~\ref{th:NK}, but instead of applying it for a fixed $\lambda$ on the space $\cX = \cY = \R$, we will employ it on the space $\cX = \cY = \ell^1_\omega(\R)$ representing real-valued functions of $\lambda$.

To that end, let us introduce the map $\cF:\ell^1_\omega(\R) \to \ell^1_\omega(\R)$, given by
\begin{equation}
\cF(x) = x^{*3} - e_1 - 2e_0, \qquad \text{for all } x\in\ell^1_\omega(\R),
\end{equation}
with $e_0$ and $e_1$ as in~\eqref{eq:identity}.
This is precisely the original zero-finding problem $F$ from~\eqref{eq:cubic} lifted in the space $\ell^1_\omega(\R)$, in the sense that
\begin{equation}
[ \cF(x) ](\lambda) = F(x(\lambda),\lambda) = x(\lambda)^3 - \lambda - 2, \qquad \text{for all } \lambda \in [-1, 1].
\end{equation}
In particular, a zero $x^\star\in\ell^1_\omega(\R)$ of $\cF$ is equivalent to a family of solutions of~\eqref{eq:cubic}.

Applying Theorem~\ref{th:NK} to the map $\cF$ requires obtaining an appropriate approximate solution $\bx$ in $\ell^1_\omega(\R)$ and a suitable operator $\cA$ in $\B(\ell^1_\omega(\R), \ell^1_\omega(\R))$. 
In fact, we can follow similar steps as in Section~\ref{sec:sqrt_C0}.
Indeed, the function $\lambda \mapsto\bx(\lambda)$ constructed there was built as a polynomial in the Chebyshev basis, hence it directly gives an element of $\ell^1_\omega(\R)$ (see Remark~\ref{rem:notation}). Similarly, $\lambda \mapsto A(\lambda)$ constructed in Section~\ref{sec:sqrt_C0} belongs to $\ell^1_\omega(\R)$ by construction, so that we then take $\cA = \cM(A)$.
That the approximate inverse $\cA$ is defined as a multiplication operator is to be expected since $D\cF(\bx)$ itself is a multiplication operator on $\ell^1_\omega(\R)$ acting as $[D\cF(\bx)]h = 3\bx^{\ast 2} \ast h$.

Then, we are left with deriving $Y, Z_1, Z_2$ for our choice of $\cF$, $\bx$ and $\cA$.
We have
\begin{align*}
\left\| \cA \cF(\bx)\right\|_{\ell^1_\omega(\R)} &= \left\| \cM(A) (\bx^{*3} - e_1 - 2e_0) \right\|_{\ell^1_\omega(\R)} = \left\| A * (\bx^{*3} - e_1 - 2e_0) \right\|_{\ell^1_\omega(\R)}, \\
\left\| I - \cA D\cF(\bx)\right\|_{\B(\ell^1_\omega(\R),\ell^1_\omega(\R))} &= \left\| I - 3\cM(A) \cM({\bx^{*2}})\right\|_{\B(\ell^1_\omega(\R),\ell^1_\omega(\R))} = \left\| e_0-3A * \bx^{*2} \right\|_{\ell^1_\omega(\R)},
\end{align*}
where we used~\eqref{eq:opnormmult} to replace the operator norm by the $\ell^1_\omega$-norm. Moreover, for any $\rstar>0$ and all $x\in \cB_{\rstar}(\bx)$,
\begin{align*}
\left\| \cA \left( D\cF(x) - D\cF(\bx)\right)\right\|&_{\B(\ell^1_\omega(\R),\ell^1_\omega(\R))} \\
&\le  \left\| \cM(A) \right\|_{\B(\ell^1_\omega(\R),\ell^1_\omega(\R))} \left( \sup_{x\in\cB_{\rstar}(\bx)}\left\| D^2\cF(x)\right\|_{\cBL(\ell^1_\omega(\R),\ell^1_\omega(\R))} \right)\left\| x-\bx \right\|_{\ell^1_\omega(\R)} \\ 
&\le 6 \left\| A  \right\|_{\ell^1_\omega(\R)} \left(\left\| \bx  \right\|_{\ell^1_\omega(\R)} + \rstar\right)\left\| x-\bx \right\|_{\ell^1_\omega(\R)},
\end{align*}
where $\cB_{\rstar}(\bx)$ denotes the closed ball centered at $\bx$ with radius $\rstar$ in $\ell^1_\omega(\cX)$, and $\cBL(\ell^1_\omega(\R),\ell^1_\omega(\R))$ is the space of bounded bilinear operators from $\ell^1_\omega(\R)\times \ell^1_\omega(\R)$ into $\ell^1_\omega(\R)$.
Therefore, we can set the bounds $Y$, $Z_1$ and $Z_2$ identical to those in~\eqref{eq:cubic_bounds_C0}.
\begin{note}[Injectivity of $\cA$]
Once more, if $Z_1<1$, then $\cA$ is automatically injective. Indeed, we then have that $A(\lambda)\neq 0$ for all $\lambda$, which, by Wiener's $1/f$ Theorem, is equivalent to $A$ being invertible in $\ell^1_\omega(\R)$, and thus $\cA = \cM(A)$ is also invertible in $\B(\ell^1_\omega(\R),\ell^1_\omega(\R))$.
\end{note}

The important difference with Section~\ref{sec:sqrt_C0} is that, if the conditions~\eqref{eq:condr} hold, then Theorem~\ref{th:NK} (applied to the map $\cF$ on $\ell^1_\omega(\R)$) yields the existence of a solution curve $x^\star$ in $\ell^1_\omega(\R)$ near $\bx$:
\begin{equation}\label{eq:err_ell1}
\left\| x^\star - \bx \right\|_{\ell^1_\omega(\R)} \le r.
\end{equation}
The above is a stronger control than the one obtained in~\eqref{eq:err_C0}.
In particular, if the weight sequence $\omega$ was taken of the form~\eqref{eq:analytic_weights} with $\mu>1$ (resp. of the form~\eqref{eq:Ck_weights} with $k\ge 2$), then~\eqref{eq:err_ell1} combined with Lemma~\ref{lem:derivatives_analytic} (resp. Lemma~\ref{lem:derivatives_Ck}) immediately implies an explicit error bound on $\frac{\mathrm{d}}{\mathrm{d}\lambda}(x^\star - \bx)$.

\begin{figure}[ht]
  \begin{center}
    \includegraphics{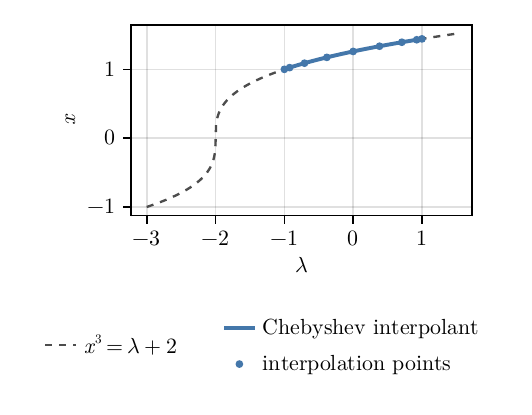}
  \end{center}
  \caption{A piece of the cubic root branch (in blue) rigorously validated in Section~\ref{sec:sqrt_ell1}.}
  \label{fig:cbrt}
\end{figure}

\paragraph{Results of the code.}
For $N=8$, the execution of the code in \cite{CODE} produces an approximate solution $\bx$ (a polynomial of degree $8$ written in the Chebyshev basis, and represented on Figure~\ref{fig:cbrt}), and, using the constant weights $\omega_n=1$ for all $n\in\N$, certifies the following upper bounds
\begin{align*}
\|\mathcal{A} \mathcal{F}(\bx)\|_{\ell^1_\omega(\R)} &\le  6.03613 \times 10^{-7} , \\
\|\mathcal{I} - \mathcal{A} D \mathcal{F}(\bar{x}) \|_{\ell^1_\omega(\R)} &\le 1.03029 \times 10^{-5}, \\
\sup_{x \in \cB_{\rstar}(\bx)}\left\| \cA \left( D\cF(x) - D\cF(\bx)\right)\right\|_{\ell^1_\omega(\R)} &\le 2.96415, \qquad \text{with } \rstar = 10^{-6},
\end{align*}
which yield a valid error bound for the entire branch of $r = 6.0362 \times 10^{-7}$.

\subsubsection{First conclusions}
\label{sec:take_home_message}

The main observation to make is how closely the bounds obtained in~\eqref{eq:cubic_bounds_C0}, used for validating the entire branch, resemble the ones initially derived in~\eqref{eq:cubic_bounds_circ}, used for validating a solution at a fixed parameter value.
Loosely speaking,
\begin{center}
\em
``The uniform estimates \eqref{eq:cubic_bounds_C0} are identical to the punctual ones \eqref{eq:cubic_bounds_circ}, \\up to a change on the product and the norm.''
\end{center}
This remark highlights an important and elegant feature of our method: the passage to uniform control involves only \emph{structural adjustments, rather than fundamentally new estimates}.
The example~\eqref{eq:cubic} was admittedly very simple, and obtaining the bounds in this case was straightforward.
Dealing with non-polynomial nonlinearities would have already required more care.
Though, we show in Section~\ref{sec:punct2unif} that this philosophy remains valid in a broader context.

Another interesting aspect is that we used exactly the same bounds $Y, Z_1, Z_2$ to apply Theorem~\ref{th:NKpara} in Section~\ref{sec:sqrt_C0} as we did to apply Theorem~\ref{th:NK}, with $\mathcal{F}$ on $\ell^1_\omega(\cX)$, in Section~\ref{sec:sqrt_ell1}.
In Section~\ref{sec:ell12C0} we show that, in general, the $\ell^1_\omega$ estimates always imply a uniform contraction, with significant implications.

\subsection{General framework}
\label{sec:framework}

Building up on the ideas and comments made on the previous example, we present our rigorous continuation setup for general Banach spaces $\cX, \cY$ and $F : \cX \times [-1,1] \to \cY$ a $C^0$ map, continuously differentiable with respect to the first variable.

We consider the continuation problem~\eqref{eq:Fxl}, that is $F(x,\lambda)=0$.
To generalize the approach followed in Section~\ref{sec:sqrt_ell1}, i.e., to recast the continuation problem on $\ell^1_\omega(\cX)$, we recall the notion of \emph{superposition operators}.

\begin{definition}\label{def:extension}
Let $\cX, \cY$ be Banach spaces.
We say that a map $\cF:\ell^1_\omega(\cX)\to\ell^1_\omega(\cY)$ is a superposition operator if there exists $F : \cX\times [-1,1] \to \cY$ such that
\begin{equation}
[\cF(x)](\lambda) = F(x(\lambda),\lambda),
\end{equation}
for all $x\in\ell^1_\omega(\cX)$ and $\lambda\in [-1,1]$.
\end{definition}

Typical continuation problems can naturally be formulated by superposition operators.
For instance, this is true for any algebraic equation, such as the cubic root example from Section~\ref{sec:cubic_root}.
On the contrary, this definition would not hold, e.g., if $F$ involves a derivative, or an integral, with respect to the parameter $\lambda$.

\begin{remark}
Consider Banach spaces $\cX$ and $\cY$.
Any element $A \in \ell^1_\omega(\B(\cX,\cY))$ gives rise to a map
\begin{equation}\label{eq:mult_sup1}
\left\{
\begin{aligned}
\cX \times [-1,1] &\longrightarrow \cY \\
(x,\lambda) &\longmapsto A(\lambda)x.
\end{aligned}
\right.
\end{equation}
The multiplication operator $\cA\bydef \cM(A)\in\B(\ell^1_\omega(\cX),\ell^1_\omega(\cY))$ introduced in Section~\ref{sec:basic_stuff_Cheb} is then the superposition operator associated to the above map.
We may simply write that $\cA$ is the superposition operator associated to $A$.


\end{remark}

As illustrated in Section~\ref{sec:sqrt_ell1}, the superposition operator is in fact the main object we work on. We show below that some properties of a superposition operator must automatically be true for the underlying map $F$.

\begin{lemma}\label{lem:prop_superposition}
Let $\cF : \ell^1_\omega(\cX) \to \ell^1_\omega(\cY)$ be the superposition operator associated with $F : \cX \times [-1,1] \to \cY$.
The following holds:
\begin{enumerate}
\item If $\cF$ is Lipchitz continuous with Lipschitz constant $\kappa \ge 0$, then $F$ is Lipschitz continuous with respect to the first variable, and its Lipschitz constant is bounded above by $\kappa$.
\item For any $x \in \ell^1_\omega(\cX)$, if the Fr\'echet derivative $D\cF(x)$ exists, then $D_x F(x(\lambda), \lambda)$ exists for all $\lambda \in [-1,1]$.
Moreover, $D\cF(x)$ is then a superposition operator such that $D\cF(x) = \cM(D_x F(x(\cdot), \cdot) )$, that is to say
\[
[D\cF(x)\psi](\lambda) = [D_x F(x(\lambda),\lambda)]\psi(\lambda),
\]
for all $\psi \in\ell^1_\omega(\cX)$ and $\lambda\in [-1,1]$.
\end{enumerate}
\end{lemma}

\begin{proof}
For any $\psi \in \cX$, denote the constant function $c_\psi \in \ell^1_\omega(\cX)$ given by $c_\psi(\lambda) = \psi$ for all $\lambda \in [-1,1]$.
Note that $\|c_\psi\|_{\ell^1_\omega(\cX)} = \|\psi\|_\cX$.

If $\cF$ is Lipschitz continuous, then, for any $\phi, \psi \in \cX$ and $\lambda \in [-1,1]$, we have that
\[
\|F(\phi, \lambda) - F(\psi, \lambda)\|_\cY
= \|\cF(c_\phi) - \cF(c_\psi)\|_{\ell^1_\omega(\cY)}
\le \kappa \| c_\phi - c_\psi \|_{\ell^1_\omega(\cX)}
= \kappa \|c_{\phi - \psi}\|_{\ell^1_\omega(\cX)}
= \kappa \|\phi - \psi\|_\cX.
\]
Next, for any $x\in\ell^1_\omega(\cX)$, $\psi \in \cX$ and $\lambda \in [-1,1]$, according to~\eqref{eq:C0vsell1} we get
\begin{align*}
\| F(x(\lambda) + \psi, \lambda) - F(x(\lambda), \lambda) - [D\cF(x)c_\psi](\lambda)\|_\cY
&\leq \| \cF(x + c_\psi) - \cF(x) - D\cF(x)c_\psi\|_{\ell^1_\omega(\cY)} \\
&= o(\|c_\psi\|_{\ell^1_\omega(\cX)}) \\
&= o(\|\psi\|_\cX),
\end{align*}
which proves that $F$ is Fr\'echet differentiable with respect to $x$ at $(x(\lambda), \lambda)$ for all $\lambda \in [-1,1]$, and that $D_x F(x(\lambda), \lambda)\psi = [D\cF(x)c_\psi](\lambda)$.
Lastly, taking now $\psi \in \ell^1_\omega(\cX)$, we readily check that
\begin{align*}
    [D\cF(x)\psi](\lambda) &= \lim_{h\to 0} \frac{[\cF(x+h\psi)](\lambda)-[\cF(x)](\lambda)}{h}\\
    &= \lim_{h\to 0} \frac{F(x(\lambda)+h\psi(\lambda),\lambda)-F(x(\lambda),\lambda)}{h}\\
    &= D_x F(x(\lambda), \lambda)\psi(\lambda). \qedhere
\end{align*}
\end{proof}

\subsubsection{Contraction in $\ell^1_\omega(\cX)$ and consequences}
\label{sec:ell12C0}

The implication of Lemma~\ref{lem:prop_superposition} is that the bounds $Y, Z_1, Z_2$ from Theorem~\ref{th:NK} estimated in $\ell^1_\omega(\cX)$, for $\cF$, $\bx$ and $\cA = \cM(A)$, also hold in $\cX$, for $F(\cdot, \lambda)$, $\bx(\lambda)$ and $A(\lambda)$ uniformly in $\lambda \in [-1,1]$.
In other words, rigorous continuation in $\ell^1_\omega(\cX)$ implies rigorous continuation in $C^0([-1,1],\cX)$.
A consequence of this fact is that no bifurcation can occur along the proven solution curve.

\begin{theorem}\label{th:nobif}
Let $\cX,\cY$ be Banach spaces, $\rstar \in (0, \infty]$, $\bx\in\ell^1_\omega(\cX)$, $\cF:\ell^1_\omega(\cX)\to\ell^1_\omega(\cY)$ a $C^1$ superposition operator associated with $F : \cX \times [-1,1] \to \cY$, and $\cA = \cM(A) : \ell^1_\omega(\cY) \to \ell^1_\omega(\cX)$ a linear operator associated with $A : [-1,1] \to \B(\cY, \cX)$, where $A(\lambda)$ is injective for all $\lambda\in[-1,1]$.
Assume that there exist positive constants $Y, Z_1, Z_2 = Z_2(\rstar)$ satisfying
\begin{subequations}\label{eq:ell_bounds}
\begin{align}
\| \cA\cF(\bx) \|_{\ell^1_\omega(\cX)} &\le Y, \\
\| \mathcal{I} - \cA D\cF(\bx) \|_{\B(\ell^1_\omega(\cX),\ell^1_\omega(\cX))} &\le Z_1, \\
\| \cA(D\cF(x) - D\cF(\bx)) \|_{\B(\ell^1_\omega(\cX),\ell^1_\omega(\cX))} &\le Z_2 \| x-\bx \|_{\ell^1_\omega(\cX)}, \qquad \text{for all } x \in \cB_{\rstar}(\bx).
\end{align}
\end{subequations}
Here, $\mathcal{I}$ denotes the identity on $\ell^1_\omega(\cX)$.
If there exists $r \in (0,\rstar]$ such that the conditions~\eqref{eq:condr} are met for the bounds $Y, Z_1, Z_2$ herein, then there exists a unique zero $x^\star$ of $\cF$ in $\cB_r(\bx)$.
Moreover, for all $\lambda \in [-1,1]$, $x^\star(\lambda)$ is the unique zero of $F(\,\cdot\,,\lambda)$ in $B_r(\bx(\lambda))$.
In particular, there is no other branch of solutions of $F$ bifurcating from $x^\star$ in $\cX$.
\end{theorem}

\begin{proof}
First note that, since $A(\lambda)$ is injective for all $\lambda\in[-1,1]$, then $\cA$ must also be injective. 
The existence of a unique $x^\star\in\cB_r(\bx)$ such that $\cF(x^\star)=0$ is then given by Theorem~\ref{th:NK} applied to $\cF:\ell^1_\omega(\cX)\to\ell^1_\omega(\cY)$ with approximate solution $\bx\in\ell^1_\omega(\cX)$ and approximate inverse $\cA :\ell^1_\omega(\cY) \to \ell^1_\omega(\cX)$.

The remaining of the proof shows that the bounds $Y, Z_1, Z_2$ given in~\eqref{eq:ell_bounds} verify the contraction conditions uniformly in $\lambda$, as required by Theorem~\ref{th:NKpara}.

Since $\cF$ and $\cA$ are superposition operators associated with $F$ and $A$ respectively, we get, for all $\lambda\in[-1,1]$,
\[
A(\lambda)F(\bx(\lambda),\lambda) = [\cA\cF(\bx)](\lambda).
\]
Hence, from~\eqref{eq:C0vsell1} it follows that
\[
\| A(\lambda)F(\bx(\lambda),\lambda) \|_\cX
\le \| \cA\cF(\bx) \|_{\ell^1_\omega(\cX)}.
\]

From Point 2 of Lemma~\ref{lem:prop_superposition}, for all $\lambda\in[-1,1]$ and $\psi\in\ell^1_\omega(\cX)$, we have
\[
(I - A(\lambda)D_xF(\bx(\lambda),\lambda))\psi(\lambda) = [(\mathcal{I} - \cA D\cF(\bx))\psi ](\lambda).
\]
For any $\psi \in \cX$, denote by $c_\psi \in \ell^1_\omega(\cX)$ the constant function given by $c_\psi(\lambda) = \psi$ for all $\lambda \in [-1, 1]$.
Using once more~\eqref{eq:C0vsell1}, we obtain
\begin{align*}
\| (I - A(\lambda)D_xF(\bx(\lambda),\lambda)) \psi \|_\cX
&= \| [(\mathcal{I} - \cA D\cF(\bx)) c_\psi](\lambda) \|_\cX \\
&\le \| (\mathcal{I} - \cA D\cF(\bx)) c_\psi \|_{\ell^1_\omega(\cX)} \\
&\le \| \mathcal{I} - \cA D\cF(\bx) \|_{\B(\ell^1_\omega(\cX),\ell^1_\omega(\cX))} \| c_\psi \|_{\ell^1_\omega(\cX)} \\
&= \| \mathcal{I} - \cA D\cF(\bx) \|_{\B(\ell^1_\omega(\cX),\ell^1_\omega(\cX))} \| \psi \|_\cX.
\end{align*}
which proves that
\[
\| I - A(\lambda)D_x F(\bx(\lambda),\lambda) \|_{\B(\cX,\cX)}
\le \| \mathcal{I} - \cA D\cF(\bx) \|_{\B(\ell^1_\omega(\cX),\ell^1_\omega(\cX))}.
\]

Finally, from Point 1 of Lemma~\ref{lem:prop_superposition}, we have readily that, for all $\lambda \in [-1,1]$,
\[
\sup_{\substack{x\in B_{\rstar}(\bx(\lambda)) \\ x \ne \bx(\lambda) }} \frac{\| A(\lambda)(D_xF(x,\lambda)- D_xF(\bx(\lambda),\lambda))\|_{\B(\cX,\cX)}}{\| x - \bx(\lambda)\|_\cX} \le \sup_{\substack{x\in\cB_{\rstar}(\bx) \\ x \ne \bx}} \frac{\| \cA(D\cF(x)- D\cF(\bx))\|_{\B(\ell^1_\omega(\cX),\ell^1_\omega(\cX))}}{\| x - \bx\|_{\ell^1_\omega(\cX)}}.
\]

Therefore, for any $\lambda\in[-1,1]$, we can also apply Theorem~\ref{th:NK} to $F(\cdot,\lambda) : \cX \to \cY$ with approximate solution $\bx(\lambda) \in \cX$ and approximate inverse $A(\lambda) \in \B(\cY, \cX)$.
Thus, there is a unique zero of $F(\cdot,\lambda)$ in $B_r(\bx(\lambda))$, which must coincide with $x^\star(\lambda)$.
\end{proof}

\begin{note}[Injectivity of $\cA$]
\label{rem:Ainj}
We have seen in the simple example of Section~\ref{sec:cubic_root} that the
injectivity of $A(\lambda)$ for each $\lambda$ can be obtained as a by-product of the bound $Z_1<1$ (implied by~\eqref{eq:condr2NK}) which makes $A(\lambda)D_xF(\bx(\lambda),\lambda)$ invertible for each $\lambda$.
When $\cX$ is infinite-dimensional, this a priori only yields the surjectivity of $A(\lambda)$ (and the injectivity of $D_xF(\bx(\lambda),\lambda)$); however, in
many examples of interest $A(\lambda)$ and $D_xF(\bx(\lambda),\lambda)$ are Fredholm operators of index $0$, in which case~\eqref{eq:condr2NK}
does yield the injectivity of $A(\lambda)$.

Therefore in practice to apply Theorem~\ref{th:NK} for $\cF$ and conduct rigorous continuation in $\ell^1_\omega(\cX)$ (assuming that surjectivity of $A(\lambda)$ contrives its injectivity), the condition $Z_1 < 1$ (implied by~\eqref{eq:condr2NK}) yields that $A(\lambda)$ is injective and in turn that $\cA = \cM(A)$ is itself injective as required.
\end{note}

\begin{remark}
The converse implication fails.
The injectivity of $\cM(A)$ does not imply that of every $A(\lambda)$.
For instance, if $A(\lambda)=\lambda I$, then any $\bu\in\ell^1_\omega(\cX)$ with $\cM(A)\bu=0$ satisfies $\lambda \bu(\lambda)=0$ for all $\lambda$, and hence $\bu=0$ by continuity, so $\cM(A)$ is injective even though $A(0)=0$.
\end{remark}

We draw the reader's attention to the fact that a failure to verify Theorem \ref{th:nobif}, i.e., the inability to find an $r$ satisfying the inequalities~\eqref{eq:condr}, may (or not!) signal a bifurcation.
If a bifurcation is indeed present, alternative strategies must be employed.
For example, one could attempt working in a smaller subspace enforcing symmetries of the solution, or resorts to other techniques such as blow-up methods \cite{AriGazKoc21,BerLesQue21,CalGarHenLesMir24}, or introduce extended zero-finding problems allowing to directly study bifurcation points~\cite{LesSanWan17,RizSanWan24}.

\subsubsection{Practical considerations}
\label{sec:punct2unif}

We have seen that a powerful framework to solve $F(x, \lambda) = 0$ on $\cX\times\Lambda$ is to formulate the problem as $\cF(x) = 0$ on $\ell^1_\omega(\cX)$ and compute the bounds in~\eqref{eq:ell_bounds}.
The philosophy of our strategy is that if one knows how to solve the problem $F(x, \lambda)= 0$ at fixed parameter value $\lambda$, then one can readily solve the continuation problem, as already illustrated in Section~\ref{sec:cubic_root} when comparing~\eqref{eq:cubic_bounds_circ} and~\eqref{eq:cubic_bounds_C0}.

The proximity between the pointwise estimates and the continuation estimates can also be made apparent in the following more general context.
For fixed parameters, many computer-assisted proofs in dynamical systems are already done using some weighted $\ell^1$ spaces for $\cX$.
Indeed, this was already case in the proof of the Feigenbaum conjecture~\cite{Lan82}, in which the fixed-point of the period-doubling renormalization operator was proven to exist as a Taylor series with prescribed geometrical decay.
These weighted $\ell^1$ sequence spaces can be used to represent a variety of objects, such as periodic solutions via Fourier series~\cite{AriKocTer05,HunLesMir16,BerBreLesVee21}, local invariant manifolds via Taylor series~\cite{Mir17,MirRei19,HenLesMir22}, or solutions to boundary value problems using Chebyshev series~\cite{LesRei14,BerShe21}.
All these representations can also be combined, for instance to study and validate connecting orbits~\cite{BerHenLes23,MirMur25,BerDucLes25}.

In the remainder of this section, we suppose for simplicity that
\begin{align*}
\cX &= \left\{ \psi \in \R^\Z \, : \, \|\psi\|_\cX = \sum_{k \in \Z} | \psi_k| <\infty  \right\}, 
\end{align*}
but the discussion below applies to any context in which $\cX$ can be identified with a weighted $\ell^1$ space with a Banach algebra structure.
As already mentioned, a convenient way to describe a one-parameter family of elements of $\cX$ is to use the space $\ell^1_\omega(\cX)$, but another natural viewpoint could be to use
\begin{align*}
\cX(\ell^1_\omega(\R)) &= \left\{ \psi \in (\ell^1_\omega(\R))^\Z \, : \, \|\psi\|_{\cX(\ell^1_\omega(\R))} = \sum_{k \in \Z} \| \psi_k \|_{\ell^1_\omega(\R)} <\infty  \right\},
\end{align*}
that is, to see each coefficient of an element of $\cX$ as a function of $\lambda$, represented by an element of $\ell^1_\omega(\R)$. A key observation is that, since $\cX$ is itself an $\ell^1$ space, $\ell^1_\omega(\cX)$ is isometrically isomorphic to $\cX(\ell^1_\omega(\R))$. Indeed, simply using Fubini's theorem we get
\[
\| x \|_{\ell^1_\omega(\cX)}
= \sum_{n\in\Z} \left\| x_{| n|}\right\|_{\cX} \omega_{| n|}
= \sum_{n\in\Z} \sum_{k\in\Z} | (x_{|n|})_k | \omega_{|n|}
= \sum_{k\in\Z}\sum_{n\in\Z} | (x_{|n|})_k | \omega_{|n|}
= \| x \|_{\cX(\ell^1_\omega(\R))}.
\]

%
Making this identification explicit is quite telling about how these objects relate to the case of a fixed parameter $\lambda$.
\begin{enumerate}
\item For any $\psi \in \ell^1_\omega(\cX)$, $\lambda \mapsto \psi(\lambda)$ is an (infinite-)vector valued function.
Considering its representation $\widehat{\psi} \in \cX(\ell^1_\omega(\R))$ given component-wise by $\widehat{\psi}_k (\lambda) = \big( \psi(\lambda) \big)_k$, then each entry of $\widehat{\psi}$ is a real valued function of $\lambda$, represented by an element of $\ell^1_\omega(\R)$.
Whence,
\[
\widehat{\psi} = \begin{pmatrix}
\vdots \\
\widehat{\psi}_k (\,\cdot\,) \\
\vdots
\end{pmatrix}, \qquad
\|\psi\|_{\ell^1_\omega(\cX)} = \|\widehat{\psi}\|_{\cX(\ell^1_\omega(\R))}
= \left\| \begin{pmatrix} \vdots \\ \| \widehat{\psi}_k(\,\cdot\,) \|_{\ell^1_\omega(\R)} \\ \vdots \end{pmatrix} \right\|_\cX.
\]
In particular, by going via $\|\widehat{\psi}\|_{\cX(\ell^1_\omega(\R))}$, we see that $\|\psi\|_{\ell^1_\omega(\cX)}$ can in fact be expressed very similarly to the pointwise quantity $\|\psi(\lambda)\|_{\cX}$, that is
\[
\psi(\lambda) = \begin{pmatrix}
\vdots \\
\widehat{\psi}_k (\lambda) \\
\vdots
\end{pmatrix}, \qquad
\| \psi(\lambda) \|_\cX
= \left\| \begin{pmatrix} \vdots \\ | \widehat{\psi}_k(\lambda) | \\ \vdots \end{pmatrix} \right\|_\cX.
\]
%

\item Similarly, for any $A \in \ell^1_\omega(\B(\cX, \cX))$, $\lambda \mapsto A(\lambda)$
is an (infinite-)matrix valued function.
Equivalently, one can consider $\widehat{A}$ given component-wise by $\widehat{A}_{k,l}(\lambda) = \big(A(\lambda)\big)_{k,l}$, and each entry of $\widehat{A}$ is a real valued function of $\lambda$, represented by an element of $\ell^1_\omega(\R)$.
Whence, for  $\cA = \cM(A)$,
\begin{align}
\label{eq:normAgeneral}
\widehat{A} = \begin{pmatrix}
&\vdots \\
\cdots & \widehat{A}_{k,l}(\,\cdot\,) & \cdots \\
&\vdots
\end{pmatrix}, \qquad
\|\cA\|_{\B(\ell^1_\omega(\cX), \ell^1_\omega(\cX))}
&=
\left\| \begin{pmatrix} & \vdots & \\ \cdots & \| \widehat{A}_{k,l}(\,\cdot\,) \|_{\ell^1_\omega(\R)} & \cdots \\ & \vdots & \end{pmatrix} \right\|_{\B(\cX, \cX)}.
\end{align}
This is again very similar to the pointwise case, in which we have
\[
A(\lambda) = \begin{pmatrix}
&\vdots \\
\cdots & \widehat{A}_{k,l}(\lambda) & \cdots \\
&\vdots
\end{pmatrix}, \qquad
\| A(\lambda) \|_{\B(\cX, \cX)} = \left\| \begin{pmatrix} & \vdots & \\ \cdots & | \widehat{A}_{k,l}(\lambda) | & \cdots \\ & \vdots & \end{pmatrix} \right\|_{\B(\cX, \cX)}.
\]
In order to prove~\eqref{eq:normAgeneral}, let us consider an arbitrary $\psi\in\ell^1_\omega(\cX)$ and $\phi = \cA \psi$. For any $k\in\Z$, $\widehat{\phi}_k = \sum_{l\in\Z} \widehat{A}_{k,l} \ast \widehat{\psi}_l$, where $\ast$ denotes the discrete convolution product on $\ell^1_\omega(\R)$, and thus 
\begin{align*}
    \Vert \phi \Vert_{\ell^1_\omega(\cX)} &= \Vert \widehat\phi \Vert_{\cX(\ell^1_\omega)} \\
    &\leq \sum_{k\in\Z}\sum_{l\in\Z} \Vert\widehat{A}_{k,l}\Vert_{\ell^1_\omega(\R)}  \Vert\widehat{\psi}_l\Vert_{\ell^1_\omega(\R)} \\
    &\leq \left( \sup_{l\in \Z} \sum_{k\in\Z} \Vert\widehat{A}_{k,l} \Vert_{\ell^1_\omega(\R)} \right) \Vert\widehat\psi\Vert_{\cX(\ell^1_\omega)}.
\end{align*}
Therefore, $\|\cA\|_{\B(\ell^1_\omega(\cX), \ell^1_\omega(\cX))} \leq \sup_{l\in \Z} \sum_{k\in\Z} \Vert\widehat{A}_{k,l} \Vert_{\ell^1_\omega(\R)}$, which is exactly the right-hand side of~\eqref{eq:normAgeneral}. By picking $\psi$ such that each $\widehat\psi_l$ is equal to $0$ except one which is taken to be the neutral element of $\ell^1_\omega(\R)$ (i.e., the constant function equal to $1$), we see that the inequality must in fact be an equality.
\end{enumerate}

We can see in both cases above that manipulating $\widehat{\psi}$ and $\widehat{A}$ in practice is very natural, and we drop the symbol $\,\, \widehat{}\,\,$ in the forthcoming examples from Section~\ref{sec:sqrt_example} and Section~\ref{sec:applications}.

On top of computing norms, the estimates required for Theorem~\ref{th:NK} typically also involve nonlinear terms, and (infinite) matrix/vector products. Here as well, one can essentially reduce such operations in the continuation case to their pointwise counterpart. This also means that the code required for rigorously validating a whole branch can mostly reuse the code used for a pointwise proof. More precisely, assume $\phi\in\ell^1_\omega(\cX)$ and $\psi\in\ell^1_\omega(\cX)$ are finite Chebyshev series, i.e. that they are polynomials with value in $\cX$. Then, their product $\phi * \psi \in\ell^1_\omega(\cX)$ can readily be computed using the DCT and its inverse together with products on $\cX$. Indeed, denoting $N_1$ the degree of $\phi$ and $N_2$ the degree of $\psi$, one can simply evaluate both $\phi$ and $\psi$ on a Chebyshev grid with $N_1+N_2+1$ points via the inverse DCT, multiply component-wise the obtained values (which are elements of $\cX$), and interpolate back using the DCT in order to retrieve the Chebyshev coefficients of the product $\phi * \psi$.
Specifically,
\[
\phi * \psi = \DCT \left( \left\{ \DCT^{-1}(\phi)^{(j)} \, \odot \DCT^{-1}(\psi)^{(j)} \right\}_{j=0}^{N_1 + N_2} \right),
\]
where $\odot$ denotes the product on $\cX$.

Similarly, assume again that $\psi\in\ell^1_\omega(\cX)$ is of degree $N_2$, and consider also $A\in\ell^1_\omega(\B(\cX,\cX))$ of degree $N_1$ together with the corresponding multiplication operator $\cA=\cM(A)\in \B\left(\ell^1_\omega(\cX),\ell^1_\omega(\cX)\right)$. In order to compute $\cA\psi$, one can use the same strategy of evaluating on a Chebyshev grid, calculating matrix/vector products in $\B(\cX,\cX)\times\cX$, and then interpolating back:
\[
\cA \psi = \DCT \left( \left\{ \DCT^{-1}(A)^{(j)} \,  \DCT^{-1}(\psi)^{(j)} \right\}_{j=0}^{N_1 + N_2} \right).
\]
%

Informally, in this context, we conclude that, for the execution of the computer-assisted proof of an entire branch of solutions, it suffices to modify the pointwise estimates as follows
\[
\begin{array}{c@{\quad}c@{\quad}c}
\text{Pointwise} & & \text{Continuation} \\
\hline \\
\cX & \rightsquigarrow & \cX(\ell^1_\omega (\R)) \\
(\R, \times) & \rightsquigarrow & (\ell^1_\omega(\R), *) \\
|\cdot| & \rightsquigarrow & \|\cdot\|_{\ell^1_\omega(\R)}
\end{array}
\]
and to use the DCT and its inverse to lift all finite calculations from $\cX$ to $\ell^1_\omega(\cX)$.
Compared to validating a solution for fixed parameters, there is therefore no additional work required in deriving continuation estimates, and very little additional work required for implementing them!
This is a very helpful feature of our approach, which will be illustrated further in Section~\ref{sec:applications}.

\begin{remark}\label{rem:systems}
    We assumed here that $\cX$ was a single $\ell^1$ space. If the problem under consideration is given by a system, or if we solve for some parameters, one would instead take for $\cX$ a cartesian product $\cX=\prod_{j=1}^J\cX_j$, where each $\cX_j$ is a Banach algebra (typically an $\ell^1$ space for a genuine component of the system, or $\R$ or $\C$ for a parameter). In such case, one can simply derive the estimates for Theorem~\ref{th:NK} component-wise. That is, instead of working with $\ell^1_\omega(\cX)$, one may also consider $\prod_{j=1}^J\ell^1_\omega(\cX_j)$ with $\Vert (x_1,\ldots,x_J)\Vert \bydef \sum_{j=1}^J \Vert x_j\Vert_{\ell^1_\omega(\cX_j)}$, which is isometrically isomorphic to $\ell^1_\omega(\cX)$, and then apply the above considerations to each $\ell^1_\omega(\cX_j)$. 

In particular, a linear operator $A\in\B(\cX,\cX)$ can then be split into several blocks
\begin{equation*}
    A = \begin{pmatrix}
A_{11} & \ldots & A_{1J} \\
\vdots & A_{ij} & \vdots \\
A_{J1} & \ldots & A_{JJ} \\
\end{pmatrix},
\end{equation*}
where each $A_{ij}$ belongs to $\B(\cX_j,\cX_i)$. The associated multiplication operator $\cA=\cM(A)$ written with respect to $\prod_{j=1}^J\ell^1_\omega(\cX_j)$ is then simply obtained by taking the multiplication operator associated to each block
\begin{equation*}
    \cA = \begin{pmatrix}
\cM(A_{11}) & \ldots & \cM(A_{1J}) \\
\vdots & \cM(A_{ij}) & \vdots \\
\cM(A_{J1}) & \ldots & \cM(A_{JJ})
\end{pmatrix}.
\end{equation*}
Of particular interest to us is the fact that an upper bound for the operator norm of $\cA$ can be obtained as follows
\begin{align}
\label{eq:normAgeneral_syst}
    &\left\Vert \cA \right\Vert_{\B\left(\prod_{j=1}^J\ell^1_\omega(\cX_j),\, \prod_{j=1}^J\ell^1_\omega(\cX_j)\right)} \leq  \\
    & \qquad\quad \left\Vert
    \begin{pmatrix}
        \left\Vert\cM(A_{11})\right\Vert_{\B\left(\ell^1_\omega(\cX_1),\, \ell^1_\omega(\cX_1)\right)} & \ldots & \left\Vert\cM(A_{1J})\right\Vert_{\B\left(\ell^1_\omega(\cX_J),\, \ell^1_\omega(\cX_1)\right)} \\
        \vdots & \left\Vert\cM(A_{ij})\right\Vert_{\B\left(\ell^1_\omega(\cX_j),\, \ell^1_\omega(\cX_i)\right)} & \vdots \\
        \left\Vert\cM(A_{J1})\right\Vert_{\B\left(\ell^1_\omega(\cX_1),\, \ell^1_\omega(\cX_J)\right)} & \ldots & \left\Vert\cM(A_{JJ})\right\Vert_{\B\left(\ell^1_\omega(\cX_J),\, \ell^1_\omega(\cX_J)\right)}
    \end{pmatrix}
    \right\Vert_{\B\left(\R^J,\R^J\right)}, \nonumber
\end{align}
with $\R^J$ endowed with the $1$-norm, and where the norm of each block can be computed according to~\eqref{eq:normAgeneral}.
Concrete examples of problems with multiple components are given in Section~\ref{sec:sqrt_example} and Section~\ref{sec:skt}.
\end{remark}

\begin{remark}
For polynomial nonlinearities, estimating the bounds in Theorem~\ref{th:NK} 
is achieved via the convolution product and the associated Banach algebra; this was the case in the cubic example in Section~\ref{sec:cubic_root}, and will also be the case in the next toy problem presented in Section~\ref{sec:sqrt_example} as well as for the more advanced problems done in Section~\ref{sec:applications}.
If the nonlinear term is non-polynomial, however, this requires that we know how to control its norm in $\ell^1_\omega(\cX)$.
Several strategies exist such as combining the problem with additional differential equations modeling the nonlinearities to retrieve a system of polynomial equations~\cite{Hen21}, using Taylor series expansions~\cite{BrePay24}, or exploiting the discrete Poisson summation formula~\cite{AalBerLes25}.
\end{remark}

\begin{remark}\label{rem:approx_interp}
Combined with the result of Theorem~\ref{th:nobif}, in some specific cases, one might want to take $\bx = \DCT(\{\bx^{(j)}\}_{j=0}^N)$, i.e., to have the exact equality $\bx(\lambda^{(j)}) = \bx^{(j)}$ at all interpolation nodes.
This property is particularly useful when enforcing constraints (e.g., a symmetry) at specific nodes, see for instance~\cite[Section 4]{CalGarHenLesMir24}.
However, a subtlety arises depending on the choice of algorithm used for the Fourier transform.
If one does not opt for the DFT and instead employs, for example, the classical Cooley--Tukey radix-2 FFT, an additional frequency mode corresponding to the Nyquist frequency appears in the discrete Fourier representation.
This mode, which occurs when the number of grid points is even, cannot in principle be discarded, as it carries essential information that does not admit a symmetric counterpart in the frequency spectrum.

That being said, we can freely impose that the Nyquist frequency corresponds to a cosine mode.
Then, the trigonometric polynomial coincide with the original one (without the cosine symmetry) exclusively at the nodes.

To be concrete, suppose that the FFT algorithm requires the sample size to be a power of $2$ and that the interpolation polynomial is required to interpolate the given nodes.
We first choose the degree of the Chebyshev interpolation polynomial as $N = 2^m$.
We then sample the curve on a Chebyshev grid of size $2^m+1$, which corresponds to a full Fourier grid of size $2^{m+1}$ (with nodes $2\pi k/2^{m+1}$ for $k=0, \dots 2^{m+1}-1$).
Applying the inverse FFT yields the coefficients of the Chebyshev interpolant, which consists of $2^m + 1$ entries; the coefficient of the Nyquist frequency is halved to enforce the cosine symmetry.
\end{remark}



\section{Choice of a parameterization and extension to multi-parameter continuation}\label{sec:reparam}

Up to now, we considered situations in which the curve of solutions of the continuation problem~\eqref{eq:Fxl} could be parametrized by $\lambda \in [-1,1]$.
However, this may not always be true, particularly when the curve folds with respect to $\lambda$, in which case a pseudo-arclength parameterization can be employed.

A distinct but related issue arises for the multi-parameter continuation, i.e., when $\lambda \in \Lambda \subset \R^d$ with $d \ge 2$.
The parameter region $\Lambda$ (assumed to be simply connected and compact) need not to be a rectangular shape re-scalable to $[-1,1]^d$ under an affine transformation.
To keep working with Chebyshev series, we must find a surjective map $\theta : [-1,1]^d \to \Lambda$ and consider the continuation problem, with $u = x \circ \theta$,
\begin{equation}\label{eq:cube2Lambda}
0 = \tilde{F}(u, s) \bydef F(u, \theta(s)), \qquad \text{for all } s\in[-1,1]^d.
\end{equation}
The injectivity of the map $\theta$ may not be necessary, although in practice a bijective $\theta$ is often used.
This question of choosing a reparameterization is also present when $d=1$ and $\Lambda=[a,b]$ is an arbitrary compact interval, but in this case the bijective affine transformation $\theta(s) = a + \frac{1+s}{2}(b-a)$ is the obvious choice.
However, other alternatives might still be of interest.
Indeed, let us consider the continuation problem
\[
0 = F(x, \lambda) = x - \cos(\lambda), \qquad \lambda\in [0,\pi].
\]
With the nonlinear change of variable $\theta(s) = \arccos(s)$, the problem simply becomes 
\[
0 = \tilde{F}(u,s) = u - s, \qquad s\in [-1,1].
\]
In particular, the exact solution $s \mapsto u^\star(s)$ has a unique non-zero Chebyshev mode.
Hence, conceptually, depending on $\theta$, one may obtain a very accurate approximate curve of solutions using fewer modes (or equivalently, fewer interpolation nodes) than with the affine reparameterization.

Finding the ``optimal'' (in a sense intentionally left unspecified) reparameterization $\theta$ is out of the scope of this article.
Nevertheless, in Section~\ref{sec:pseudo_arclength}, we detail how the pseudo-arclength parameterization integrates into our rigorous continuation framework.
Next, we provide in Section~\ref{sec:multi_dim} a convenient approach to handle multi-parameter continuation when $\Lambda$ does not have a rectangular shape.

\subsection{Rigorous pseudo-arclength continuation}
\label{sec:pseudo_arclength}

In this section, we consider a one-parameter continuation problem of the form~\eqref{eq:Fxl}, where $\Lambda\subset\R$ is a compact interval.
The solution manifold cannot, in general, be represented globally as a graph $x^\star = x^\star(\lambda)$ over $\lambda \in \Lambda$.
A fold point $(x^{(fold)}, \lambda^{(fold)})$ illustrates this since $D_x F(x^{(fold)}, \lambda^{(fold)})$ ceases to be invertible so that rigorously validating a branch via Theorem~\ref{th:NK} fails.
We must then decide on a parameterization $s \in [-1, 1] \mapsto (x^\star(s), \lambda^\star(s))$ of the solution curve.
From a theoretical perspective, it is standard to use the arclength as a parameterization; although, in practice, an approximation, called \emph{pseudo-arclength}, is sufficient.

\subsubsection{Pseudo-arclength and corresponding zero-finding problem}
\label{sec:pal_method}

The pseudo-arclength continuation method consists in iteratively constructing an approximation of the solution curve using a \emph{predictor-corrector scheme} (see e.g., \cite{krauskopf2007numerical}).
Given an approximate point $\bu^{(j)} = (\bar{x}^{(j)}, \bar{\lambda}^{(j)})\in\cX\times\R$ on the curve (i.e., $F(\bu^{(j)})\approx 0$) and an associated approximate tangent vector $\dot{u}^{(j)}$ (i.e., $D_{(x, \lambda)} F (\bu^{(j)}) \dot{u}^{(j)} \approx 0$), one defines a hyperplane via the affine constraint
\begin{equation}\label{eq:hyperplanes}
\left\langle (x, \lambda) - (\bu^{(j)} + \delta^{(j)} \dot{u}^{(j)}), \dot{u}^{(j)} \right\rangle = 0,
\end{equation}
where $\delta^{(j)} > 0$ is a chosen \emph{step-size}, $\bu^{(j)} + \delta^{(j)} \dot{u}^{(j)}$ is called the \emph{predictor} and $\langle \cdot, \cdot \rangle$ is some inner product on $\cX\times\R$.
This hyperplane generically intersects the solution curve transversely, if $\delta^{(j)}$ is small enough. We then compute an approximate zero $\bu^{(j+1)} =(\bar{x}^{(j+1)}, \bar{\lambda}^{(j+1)})$ of $F$ on that hyperplane, called the \emph{corrector}.
If the tangent vector $\dot{u}^{(j)}$ was taken such that $\langle \dot{u}^{(j)}, \dot{u}^{(j)}\rangle = 1$, then $\delta^{(j)}$ corresponds to a first order approximation of the arclength of the curve between $\bu^{(j)}$ and $\bu^{(j+1)}$.

By repeating this process, we generate a collection of $N+1$ points lying approximately on the solution curve.
To retrieve an approximation of this curve, we let $s^{(j)}$ be the Chebyshev nodes in $[-1,1]$,
\begin{equation}
s^{(j)} = - \cos\Big(\frac{j\pi}{N} \Big), \qquad j = 0, \dots, N,
\end{equation}
and compute the polynomial $\bu:[-1,1]\to \cX\times\R$ such that $\bu(s^{(j)}) \approx \bu^{(j)}$ for $j = 0, \dots, N$.
With the notations of Section~\ref{sec:basic_stuff_Cheb}, $\bu$ is simply obtained as $\bu \approx \DCT^{-1}\left(\{\bu^{(j)}\}_{j=0}^N\right)$.

In order to validate the approximate curve of solutions $\bu$, we introduce the spaces $\cX_\pa = \cX\times\R$ and $\cY_\pa = \cY\times\R$, where $u=(x,\lambda)$ denotes a generic elements in $\cX_\pa$.
We also compute $\dot{u} \approx \DCT^{-1}\left(\{\dot{u}^{(j)}\}_{j=0}^N\right)$, and note that both $\bu$ and $\dot{u}$ can be seen as elements of $\ell^1_\omega(\cX_\pa)$. We then consider the zero-finding problem 
$F_\pa : \cX_\pa \times [-1,1]  \to \cY_\pa$ given by
\begin{equation}
0 = F_\pa(u, s) \bydef
\begin{pmatrix}
F ( u) \\
\langle u-\bu(s), \dot{u}(s) \rangle
\end{pmatrix}, \qquad \text{for all } s\in [-1,1].
\end{equation}
As a brief remark, we have chosen to consider a slightly different set of hyperplanes than those produced numerically by the pseudo-arclength method given in~\eqref{eq:hyperplanes}.
Either choice is valid, but we prefer this set because the last component of $F_\pa(\bu(s), s)$ then vanishes.

We are now back to the situation considered in Section~\ref{sec:parameter_continuation}, except that the parameter is now called $s$ instead of $\lambda$.
In particular, note that even if we are at a fold point $u = (x,\lambda)$ for the original $F(x,\lambda)=0$ problem, and therefore $D_x F(u)$ has a one-dimensional kernel, $D_u F_\pa(u,s)$ is still expected to be invertible (see, e.g.,~\cite[Chapter 1]{KraOsiGal07}).
Therefore, there is no longer a theoretical obstruction to rigorously validate the solution curve passing through this point.
Following Section~\ref{sec:framework}, we apply Theorem~\ref{th:NK} to an extension $\cF_\pa$ of $F_\pa$, defined on the space $\ell^1_\omega(\cX_\pa)$, and with the approximate zero $\bu\in \ell^1_\omega(\cX_\pa)$ constructed above.

\begin{remark}\label{rem:nobif_pal}
In that case, we also do not hit any bifurcation on the proven branch back in the projected space $(x, \lambda)$.
Indeed, as discussed in Section~\ref{sec:ell12C0}, a successful application of Theorem~\ref{th:NK} to $\cF$ on $\ell^1_\omega(\cX)$ (or of Theorem~\ref{th:NKpara} to $F$ on $\cX\times\Lambda$) proves that the validated branch does not exhibit bifurcations. We note that this still holds if the branch is validated using pseudo-arclength continuation, i.e., if we successfully apply Theorem~\ref{th:NK} to $\cF_\pa$ on $\ell^1_\omega(\cX_\pa)$ (or Theorem~\ref{th:NKpara} to $F_\pa$ on $\cX_\pa\times\Lambda$).
Then, condition~\eqref{eq:condr2NK} shows that $D_u F_\pa(u^\star(s),s)$ is injective and therefore that the kernel of $DF(u^\star(s))$ is at most one-dimensional, for all $s\in[-1,1]$.
\end{remark}

\begin{remark}
This procedure does not enforce how the step-size $\delta^{(j)}$ is chosen.
As such, each approximate point $\bu^{(j)}$ is located at an arclength $\frac{L}{2}(\theta^{(j)}+1)$, where $L$ denotes the total arclength, for some value $\theta^{(j)}$, along the solution curve.
Hence, we approximate the arclength parameterization well if $\theta^{(j)} \approx s^{(j)}$.

However, a standard predictor-corrector scheme would probably not choose $\delta^{(j)}$ in such a way that $\theta^{(j)} \approx s^{(j)}$. 
Yet, we believe that it is worth mimicking an arclength parameterization of the curve.
It is known to be optimal in terms of analyticity of the parameterization (see, e.g.,~\cite{NesPap17}), whereas even if the solution curve is a smooth one-dimensional manifold, an arbitrary parameterization $s \mapsto (x^\star(s),\lambda^\star(s))$ of that curve may not be smooth.
The contraction argument described in Section~\ref{sec:ell12C0} relies on the ability to to obtain an accurate approximation of the desired curve of solutions $s\mapsto(x^\star(s),\lambda^\star(s))$ using polynomial interpolation at Chebyshev nodes.
This approach is more efficient the smoother the curve.

To retrieve a close approximation of the arclength parameterization, we introduce a target arclength $L > 0$ for the piece of the curve we are interested in, and pick the successive step-sizes as
\begin{equation}\label{eq:deltaj}
\delta^{(j)} = \frac{L}{2} \left(s^{(j+1)}-s^{(j)} \right),
\end{equation}
so that indeed, each approximate point $\bu^{(j)}$ is located at an arclength close to $\frac{L}{2}(s^{(j)} + 1)$ along the curve of solutions.
The target arclength $L$ can be estimated afterwards (instead of an a priori heuristic), and one may then re-sample the curve at~\eqref{eq:deltaj}.
\end{remark}

\subsubsection{A basic example: the square root}
\label{sec:sqrt_example}

Let us consider the square root problem
\begin{equation}\label{eq:sqrt}
 0 = F(x, \lambda) = x^2 - \lambda, \qquad x\in\R,\, \lambda\ge 0.   
\end{equation}
Our goal is to compute the curve of solutions to~\eqref{eq:sqrt} for $\lambda \in [0,1]$, meaning the union of $\{ -\sqrt{\lambda} \, : \, \lambda \in [0,1] \}$ and $\{ \sqrt{\lambda} \, : \, \lambda \in [0,1] \}$.
The pseudo-arclength continuation would also be useful for the cubic root example, as it allows for a rigorous continuation through the apparent singularity at the origin, which is not possible with the continuation setup used in Section~\ref{sec:cubic_root}.
Moreover, a more complex example is detailed in Section~\ref{sec:skt}, where we validate curves of non-trivial steady-states for the Shigesada--Kawasaki--Teramoto system.

We proceed as described in Section~\ref{sec:pal_method}.
An approximate arclength from points obtained numerically along the curve would be sufficient, but for this toy problem it is easily computable by hand and we find $L=\sqrt{5} + \frac{1}{2} \ln(2 + \sqrt{5})$.
We start from the upper end of the branch at $\lambda = 1$, so we take $\bu^{(0)} = (\bx^{(0)},\bar{\lambda}^{(0)}) = (1,1)$, and compute successive approximate points $\bu^{(j)}$ using step-sizes chosen in terms of the Chebyshev nodes as in~\eqref{eq:deltaj}.
This procedure also produces approximate tangent vectors $\dot{u}^{(j)} = (\dot{x}^{(j)},\dot{\lambda}^{(j)})$, and we build polynomial functions $\bu \approx \DCT^{-1}\left(\{\bu^{(j)}\}_{j=0}^N\right)$ and $\dot{u} \approx \DCT^{-1}\left(\{\dot{u}^{(j)}\}_{j=0}^N\right)$.
The approximate solution curve $\bu$ obtained this way, as well as the corresponding hyperplanes at the nodes $s^{(j)}$, are depicted in Figure~\ref{fig:sqrt}.

The zero-finding problem for the pseudo-arclength is then $F_\pa:\R^2 \times[-1,1]\to\R^2$ given by
\[
F_\pa(u,s) = 
\begin{pmatrix}
x^2 - \lambda \\
(x-\bx(s)) \dot{x}(s) + (\lambda-\bar{\lambda}(s)) \dot{\lambda}(s)
\end{pmatrix}, \qquad \text{for all } u = (x,\lambda) \in\R^2, s\in[-1,1].
\]
Here, we consider $\cX_\pa = \R^2$ endowed with the $1$-norm
\[
\| u \|_{\R^2} \bydef |x| + |\lambda|, \qquad \text{for all } u = (x, \lambda) \in \R^2.
\]
The superposition operator $\cF_\pa : \left(\ell^1_\omega(\R) \right)^2 \to \left(\ell^1_\omega(\R) \right)^2$ associated to $F_\pa$ is given by
\[
\cF_\pa(u) = 
\begin{pmatrix}
x * x - \lambda  \\
(x-\bx) * \dot{x} + (\lambda-\bar{\lambda}) *\dot{\lambda}
\end{pmatrix}, \qquad \text{for all } u=(x,\lambda)\in\left(\ell^1_\omega(\R) \right)^2,
\]
and we consider the following norm:
\begin{equation*}
    \| u \|_{\left(\ell^1_\omega(\R) \right)^2} \bydef \| x \|_{\ell^1_\omega(\R)} + \| \lambda \|_{\ell^1_\omega(\R)}, \qquad \text{for all } u = (x, \lambda) \in \left(\ell^1_\omega(\R) \right)^2.
\end{equation*}
Note that, according to Remark~\ref{rem:systems}, we immediately defined $\cF_\pa$ on $\ell^1_\omega(\R)\times \ell^1_\omega(\R)$ instead of using the (isometrically isomorphic) space $\ell^1_\omega(\R\times\R)$.
Moreover,
we can define the approximate inverse as
\[
\cA = \cM(A) =
\begin{pmatrix}
\cM\left(A_{11}\right) & \cM\left(A_{12}\right) \\
\cM\left(A_{21}\right) & \cM\left(A_{22}\right)
\end{pmatrix} :
\left(\ell^1_\omega(\R)\right)^2 \to \left(\ell^1_\omega(\R)\right)^2,
\]
where
\[
A =
\begin{pmatrix}
A_{11} & A_{12} \\
A_{21} & A_{22}
\end{pmatrix} \approx \DCT^{-1} \left( \left\{D_{u} F_\pa(\bu^{(j)}, s^{(j)})^{-1} \right\}_{j=0}^N\right).
\]
The $\approx$ symbol emphasizes that neither the matrix inversion nor $\DCT^{-1}$ have to be computed rigorously.

In order to apply Theorem~\ref{th:NK} to the superposition map $\cF_\pa$,
we need to derive formulas to compute $Y$, $Z_1$ and $Z_2$.
We find, with $\bu = (\bx, \bar{\lambda})\in \left(\ell^1_\omega(\R) \right)^2 $,
\begin{align*}
\left\| \cA \cF_\pa(\bu) \right\|_{\left(\ell^1_\omega(\R) \right)^2}
&= \left\| \begin{pmatrix}
\cM\left(A_{11}\right) & \cM\left(A_{12}\right) \\
\cM\left(A_{21}\right) & \cM\left(A_{22}\right)
\end{pmatrix}
\begin{pmatrix}
\bx * \bx - \bar{\lambda} \\
0
\end{pmatrix} \right\|_{\left(\ell^1_\omega(\R) \right)^2} \\
&= \left\|\begin{pmatrix} \left\| A_{11}*(\bx* \bx - \bar{\lambda} ) \right\|_{\ell^1_\omega(\R)} \\ \left\| A_{21}*(\bx* \bx - \bar{\lambda} ) \right\|_{\ell^1_\omega(\R)} \end{pmatrix}\right\|_{\R^2}
=Y,
\end{align*}
\begin{align*}
&\left\| \mathcal{I} - \cA D\cF_\pa(\bu) \right\|_{\B\left(\left(\ell^1_\omega(\R) \right)^2,\,\left(\ell^1_\omega(\R) \right)^2\right)}\\
& \qquad\qquad\qquad = \left\| \begin{pmatrix}
\cM(A_{11}* 2\bx + A_{12}* \dot{x}-e_0)  & \cM(-A_{11} + A_{12}* \dot{\lambda}) \\
\cM(A_{21}* 2\bx + A_{22}* \dot{x})  & \cM(-A_{21} + A_{22}* \dot{\lambda}-e_0)
\end{pmatrix} \right\|_{\B\left(\left(\ell^1_\omega(\R) \right)^2,\,\left(\ell^1_\omega(\R) \right)^2\right)} \\
& \qquad\qquad\qquad \le \left\| \begin{pmatrix}
\| A_{11}* 2\bx + A_{12}* \dot{x} -e_0 \|_{\ell^1_\omega(\R)}  & \|-A_{11} + A_{12}* \dot{\lambda}\|_{\ell^1_\omega(\R)} \\
\|A_{21}* 2\bx + A_{22}* \dot{x}\|_{\ell^1_\omega(\R)} & \|-A_{21} + A_{22}* \dot{\lambda} -e_0\|_{\ell^1_\omega(\R)}
\end{pmatrix} \right\|_{\B(\R^2, \R^2)} = Z_1,
\end{align*}
and, for all $u=(x,\lambda)\in\cX_\pa$,
\begin{align*}
\left\| \cA(D\cF_\pa(u) - D\cF_\pa(\bu)) \right\|_{\B\left(\left(\ell^1_\omega(\R) \right)^2,\,\left(\ell^1_\omega(\R) \right)^2\right)} & = \left\| \begin{pmatrix}
\cM\left(A_{11}\right) & \cM\left(A_{12}\right) \\
\cM\left(A_{21}\right) & \cM\left(A_{22}\right)
\end{pmatrix}
\begin{pmatrix}
2\cM\left(x - \bx\right) & 0 \\
0 & 0
\end{pmatrix}\right\|_{\B\left(\left(\ell^1_\omega(\R) \right)^2,\,\left(\ell^1_\omega(\R) \right)^2\right)} \\
&\leq 2\left\| \begin{pmatrix}
\| A_{11}* (x-\bx)\|_{\ell^1_\omega(\R)} & 0 \\
\| A_{21}* (x-\bx)\|_{\ell^1_\omega(\R)} & 0
\end{pmatrix}\right\|_{\B(\R^2, \R^2)} \\
&= 2\left(\left\| A_{11}* (x-\bx)\right\|_{\ell^1_\omega(\R)} + \left\| A_{21}* (x-\bx)\right\|_{\ell^1_\omega(\R)} \right) \\
&\leq 2\left(\left\| A_{11}\right\|_{\ell^1_\omega(\R)} + \left\| A_{21}\right\|_{\ell^1_\omega(\R)} \right) \left\| u-\bu\right\|_{_{\left(\ell^1_\omega(\R) \right)^2}},
\end{align*}
hence we take $Z_2 = 2\left(\left\| A_{11}\right\|_{\ell^1_\omega(\R)} + \left\| A_{21}\right\|_{\ell^1_\omega(\R)} \right)$.

\begin{figure}[ht!]
  \begin{center}
    \includegraphics{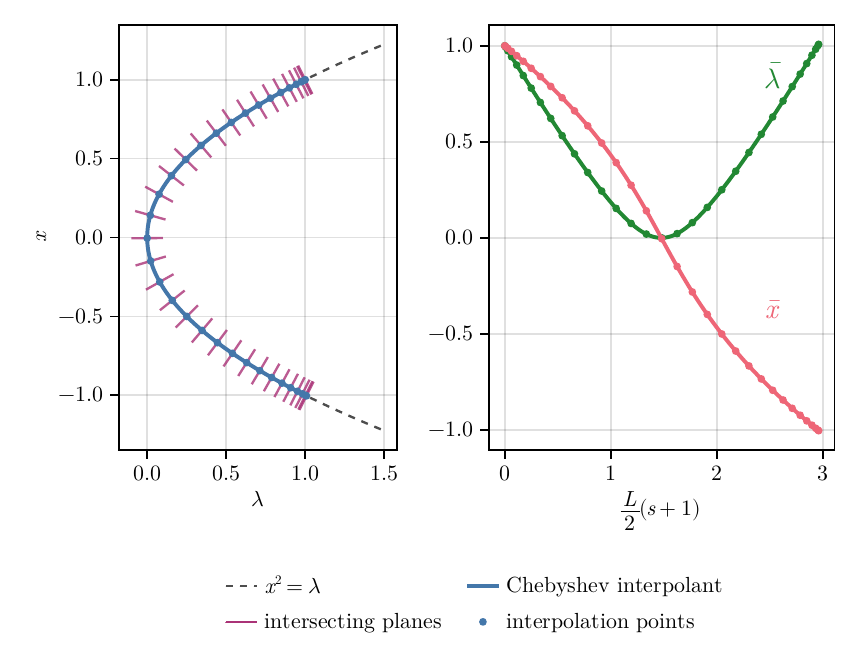}
  \end{center}
  \caption{Illustration of the rigorous validation of a piece of the square root branch based on the pseudo-arclength reformulation.}
  \label{fig:sqrt}
\end{figure}

\paragraph{Results of the code.}
For $N = 32$, the execution of the code in \cite{CODE} produces the approximate curve represented on Figure~\ref{fig:sqrt}, and certifies the following bounds
\begin{align*}
Y &=  1.81862 \times 10^{-6}, \\
Z_1 &= 2.44559 \times 10^{-4}, \\
Z_2 &= 4.98384, \qquad \text{with } \rstar = \infty,
\end{align*}
and a valid error bound for the entire branch is $r = 1.81907 \times 10^{-6}$.

\subsection{Multi-parameter continuation}
\label{sec:multi_dim}

Let us now go back to the original continuation problems of the form~\eqref{eq:Fxl}, and discuss how to extend the rigorous continuation approach presented in Section~\ref{sec:parameter_continuation} when the parameter region $\Lambda \subset \R^d$ is a higher-dimensional set ($d \geq 2$), assumed to be compact and simply connected. 

If $\Lambda = [-1, 1]^d$, this extension is completely straightforward.
Indeed, we can represent the solution manifold as a multi-variate Chebyshev series of the form
\begin{equation}
x(\lambda) = \sum_{n = (n_1, \dots, n_d) \in \Z^d} x_{|n_1|, \dots, |n_d|} T_{|n_1|} (\lambda_1) \cdots T_{|n_d|} (\lambda_d), \qquad \text{for all } \lambda = (\lambda_1, \dots, \lambda_d) \in [-1, 1]^d,
\end{equation}
starting from an approximate solution $\bx$ constructed as a truncated multi-variate Chebyshev series (i.e., a multi-variate polynomial written in the Chebyshev basis).
As in the $d=1$ case, this approximate solution can be efficiently constructed by sampling the solution manifold on the Chebyshev grid in $[-1,1]^d$ generated by the nodes
\begin{equation}\label{eq:multi_chebyshev_nodes}
\lambda^{(j_1, \dots, j_d)} = \left(-\cos \Big(\frac{j_1\pi}{N_1}\Big), \dots, -\cos \Big(\frac{j_d\pi}{N_d}\Big) \right), \quad j_l = 0, \dots, N_l, \qquad l = 1, \dots, d,
\end{equation}
and then interpolating using the multi-dimensional version of the inverse DCT given in~\eqref{eq:DCT}.

Given $\omega = \left(\omega^{(1)},\ldots,\omega^{(d)}\right)$, where each $\omega^{(j)}$ is a weight sequence satisfying~\ref{A1}-\ref{A2}, the sequence space~\eqref{eq:weighted_ell1} introduced in the $d=1$ case generalizes to
\begin{equation}\label{eq:ell1_multi}
\ell^{1,\otimes d}_\omega (\cX) \bydef \left\{ \psi \in \cX^{\N^d} \, : \, \| \psi \|_{\ell^{1,\otimes d}_\omega (\cX)} \bydef \sum_{n \in \Z^d} \left\| \psi_{|n_1|, \dots, |n_d|} \right\|_\cX \omega^{(1)}_{|n_1|} \cdots \omega^{(d)}_{|n_d|} < \infty \right\}.
\end{equation}
The discrete convolution given in \eqref{eq:conv} also generalizes in a direct way, namely, for any $\phi, \psi \in \ell^{1,\otimes d}_\omega (\C)$, we have
\begin{equation}\label{eq:conv_multi}
(\phi * \psi)_n \bydef \sum_{n' \in \Z^d} \phi_{|n_1 - n_1'|, \dots, |n_d - n_d'|} \psi_{|n_1'|, \dots, |n_d'|},
\end{equation}
and a Banach algebra property analogue to~\eqref{eq:banach_algebra} holds.
Therefore, if we are able to rigorously validate a solution of~\eqref{eq:Fxl} in $\cX$ for a fixed $\lambda_\circ\in[-1,1]^d$, then we should be able to validate the entire manifold of solutions by conducting a proof in $\ell^{1,\otimes d}_\omega (\cX)$, using an extension $\cF : \ell^{1,\otimes d}_\omega (\cX) \to \ell^{1,\otimes d}_\omega (\cY)$ of $F$, in exactly the manner described in Section~\ref{sec:framework}.

In contrast, when $\Lambda$ is not a rectangular parameter region
, we must then make a choice on how to deform $[-1,1]^d$ into $\Lambda$, as in~\eqref{eq:cube2Lambda}.
Many considerations could be made, including the use of an alternative basis that may be better suited, according to some criteria, to the geometry of $\Lambda$.
The literature on parameterization and meshing is extensive, and a detailed study on how to best choose $\theta$ would lead us too far afield from the main focus of this article. 
The important point is that, once a mapping $\theta$ has been selected, we are back to a continuation problem on $[-1,1]^d$ for which the rigorous continuation method presented in Section~\ref{sec:framework} readily applies.
To be precise, we then look for $u = x \circ \theta$ solving the continuation problem
\begin{equation}
0 = \tilde{F}(u, s) \bydef F(u, \theta(s)), \qquad \text{for all } s \in [-1,1]^d.
\end{equation}

In the remainder of this section, we present some simple choices for $\theta$ that proved sufficient for the examples treated in this paper.

We begin with a triangular parameter region $\Lambda \subset \R^2$, say the triangle with vertices $P_1 = (-1, 0)$, $P_2 = (1,-1)$ and $P_3 = (1,1)$.
In the $O\lambda_1\lambda_2$ coordinate system, the lower segment $\overline{P_1 P_2}$ and upper one $\overline{P_1 P_3}$ are graphs over the $\lambda_1$-axis, given by $g_+(\lambda_1) = \frac{1+\lambda_1}{2}$ and $g_- = -g_+$, respectively.
Thus, an elementary way of mapping the square $[-1,1]^2$ to this triangle is via a linear interpolation from $g_-$ to $g_+$:
\begin{equation}\label{eq:trianglemap1}
\theta : (s_1, s_2) \mapsto \left(s_1, g_-(s_1) + \frac{1+s_2}{2}(g_+(s_1) - g_-(s_1)) \right) = \left(s_1, s_2 \frac{1 + s_1}{2} \right).
\end{equation}
Figure~\ref{fig:cheb_theta} illustrates this choice of $\theta$ for a Chebyshev grid generated by the nodes~\eqref{eq:multi_chebyshev_nodes} with $N_1 = N_2 = 16$.
We can then proceed as in~\eqref{eq:cube2Lambda}, and recover a continuation problem set on the square $[-1,1]^2$.

\begin{figure}
    \centering
\begin{tikzpicture}[scale=2]

\begin{scope}
  \draw[->] (-1.2,0) -- (1.2,0) node[right] {$s_1$};
  \draw[->] (0,-1.2) -- (0,1.2) node[above] {$s_2$};

  \draw[thick] (-1,-1) rectangle (1,1);

  \draw[very thick,blue] (-1,-1) -- (-1,1);

  \fill[blue] (-1,-1) circle (0.075);
  \fill[blue] (-1, 1) circle (0.075);
  \fill ( 1,-1) circle (0.075);
  \fill ( 1, 1) circle (0.075);

  \node[below left, blue]  at (-1,-1) {$(-1,-1)$};
  \node[below right] at (1,-1)  {$(1,-1)$};
  \node[above left, blue]  at (-1,1)  {$(-1,1)$};
  \node[above right] at (1,1)   {$(1,1)$};

\foreach \i in {0,...,8} {
  \foreach \j in {0,...,8} {
    \pgfmathsetmacro{\sone}{-cos(\i*pi/8 r)}
    \pgfmathsetmacro{\stwo}{-cos(\j*pi/8 r)}
    \fill[purple] (\sone,\stwo) circle (0.035);
  }
}
\end{scope}

\draw[thick,->] (2,0) -- (2.5,0) node[midway,above] {$\theta$};

\begin{scope}[shift={(4,0)}]
  \draw[->] (-1.2,0) -- (1.2,0) node[right] {$\lambda_1$};
  \draw[->] (0,-1.2) -- (0,1.2) node[above] {$\lambda_2$};

  \coordinate (P1) at (-1,0);
  \coordinate (P2) at (1,-1);
  \coordinate (P3) at (1,1);

  \draw[thick] (P1)--(P2)--(P3)--cycle;

  \foreach \p in {P2,P3} \fill (\p) circle (0.075);
  \fill[blue] (P1) circle (0.075);

  \node[below, blue] at (-1,-0.4) {$P_1=(-1,0)$};
  \node[below right] at (P2) {$P_2=(1,-1)$};
  \node[above right] at (P3) {$P_3=(1,1)$};

\foreach \i in {0,...,8} {
  \foreach \j in {0,...,8} {
    \pgfmathsetmacro{\sone}{-cos(\i*pi/8 r)}
    \pgfmathsetmacro{\stwo}{-cos(\j*pi/8 r)}
    
    \pgfmathsetmacro{\lx}{\sone}              
    \pgfmathsetmacro{\ly}{\stwo*(1+\sone)/2}  
    
    \fill[purple] (\lx,\ly) circle (0.035);
  }
} 
\end{scope}

\end{tikzpicture}
\caption{Chebyshev grid of size $(N_1+1) \times (N_2+1) = (2^3+1) \times (2^3+1) = 9 \times 9$ in the square $[-1,1]^2$ (left) and its image onto the triangle $\Lambda = (P_1, P_2, P_3)$ under the map~\eqref{eq:trianglemap1} (right).}
\label{fig:cheb_theta}
\end{figure}
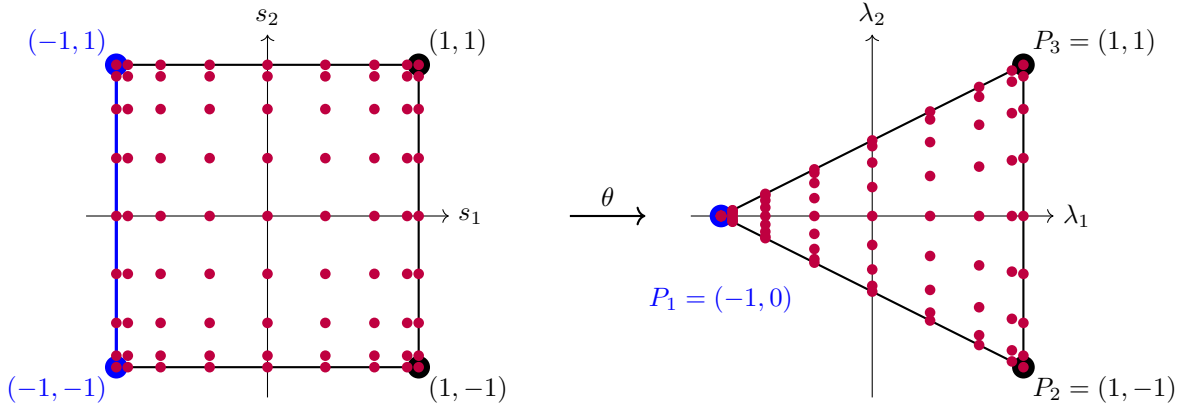

That said, the map~\eqref{eq:trianglemap1} has an unintended bias as the entire edge $\{s_1 = -1\}$ of $[-1, 1]^2$ is mapped to $P_1$.
Hence, an important proportion of the Chebyshev grid is concentrated around $P_1$.
Without distracting ourselves with a thorough investigation, we could attempt to alleviate this accumulation by adding a fourth point $P_4$ on the boundary of the triangle.
These considerations lead us to Coons' parameterization method~\cite{Coo64,Far14}.
Suppose that there are four surjective boundary maps $\gamma_1, \gamma_2, \gamma_3, \gamma_4 : [-1, 1] \mapsto C_1, C_2, C_3, C_4$ such that $\partial\Lambda = \bigcup_{j=1}^4 C_j$ and
\begin{subequations}\label{eq:bdy_rel}
\begin{alignat}{4}
&C_1 \cap C_4 = \{P_1\}, \qquad &&C_1 \cap C_2 = \{P_2\}, \qquad &&C_2 \cap C_3 = \{P_3\}, \qquad &&C_3 \cap C_4 =\{P_4\}, \\
&P_1 = \gamma_1(-1) = \gamma_4(1), \qquad &&P_2 = \gamma_1(1) = \gamma_2(-1), \qquad &&P_3 = \gamma_2(1) = \gamma_3(-1) \qquad &&P_4 = \gamma_3(1) = \gamma_4(-1).
\end{alignat}
\end{subequations}
We define the map $\theta : [-1, 1]^2 \to \R^2$ by
\begin{equation}\label{eq:coons}
\theta : (s_1, s_2) \mapsto \theta_{1\to3} (s_1, s_2) + \theta_{4\to2}(-s_2, s_1) - \theta_\text{corners}(s_1, s_2),
\end{equation}
where, denoting $\alpha(s) = \frac{1 + s}{2}$,
\begin{subequations}
\begin{align}
\theta_{i\to j} (s_1, s_2) &= \gamma_i( s_1 ) + \alpha(s_2) \big(\gamma_j(-s_1 ) - \gamma_i( s_1 )\big), \\
\theta_\text{corners} (s_1, s_2) &= P_1 + \alpha(s_1) (P_2 - P_1) + \alpha(s_2) \big(P_4 - P_1 + \alpha(s_1) (P_3 - P_4 - P_2 + P_1) \big).
\end{align}
\end{subequations}

Under some assumptions on the boundary curves, Coons' parameterization can be guaranteed to $(i)$ cover the desired parameter region $\Lambda$, $(ii)$ be onto $\Lambda$, and $(iii)$ define a global diffeomorphism from $(-1,1)^2$ onto $\Lambda$.
This follows from the next proposition.

\begin{proposition}\label{prop:coons_onto}
Let $\Lambda$ be a compact and simply connected subset of $\R^2$ such that $\partial \Lambda$ is a Jordan curve, and $D = [-1, 1]^2$.
If $\theta : D \to \R^2$ is continuous, and $\theta(\partial D) = \partial \Lambda$, then $\theta(D) \supset \Lambda$.

If, additionally, $\theta$ is a local diffeomorphism everywhere in the interior $\mathring{D}$ of $D$, then $\theta(D) = \Lambda$.
Further assuming that $\theta|_{\partial D}$ covers $\partial \Lambda$ only once implies that $\theta$ is a diffeomorphism from $\mathring{D}$ onto the interior $\mathring{\Lambda}$ of $\Lambda$.
\end{proposition}

\begin{proof}
The Jordan curve theorem implies that $\mathring{\Lambda}$ and $\Lambda^c$ can be characterized in terms of the winding number $\mathrm{wind}(\theta|_{\partial D}, \cdot)$ as follows
\[
\mathring{\Lambda} = \{ y \in \R^2 \smallsetminus \partial \Lambda \, : \, \mathrm{wind}(\theta|_{\partial D}, y) \ne 0 \}, \qquad
\Lambda^c = \{ y \in \R^2 \smallsetminus \partial \Lambda \, : \, \mathrm{wind}(\theta|_{\partial D}, y) = 0 \}.
\]
Then, for any $y \in \mathring{\Lambda}$, we have that the Brouwer degree $\mathrm{deg}(\theta, \mathring{D}, y)$ of $\theta$ is equal to $\mathrm{wind}(\theta|_{\partial D}, y) \ne 0$.
Hence, there exists $x \in \mathring{D}$ such that $\theta(x) = y$.
Since $\theta(\partial D) = \partial \Lambda$, this proves that $\theta(D)$ contains $\Lambda$.

To prove the converse inclusion when $\theta$ is a local diffeomorphism, we argue by contradiction, and suppose $\theta(D)\cap \Lambda^c \neq \emptyset$. Since $\Lambda^c$ is connected and unbounded whereas $\theta(D)$ is bounded, this means that $\partial(\theta(D))\cap \Lambda^c \neq \emptyset$, and we now consider $y\in \partial(\theta(D))\cap \Lambda^c$. Using that $\theta(D)$ is compact, we have $y\in\theta(D)$, and the assumption $\theta(\partial D) = \partial\Lambda \subset \Lambda$ yields that $y\in \theta(\mathring{D})$. Therefore, we must have $y\in \partial(\theta(\mathring{D}))$ (otherwise, $y$ would be in the interior of $\theta(\mathring{D})$, hence in the interior of $\theta(D)$, which is incompatible with $y \in \partial(\theta(D))$). However, by assumption $\theta$ is a local diffeomorphism on $\mathring{D}$, therefore $\theta(\mathring{D})$ is open, and $\theta(\mathring{D}) \cap \partial(\theta(\mathring{D})) = \emptyset$, contradiction. 

Lastly, if $\theta|_{\partial D}$ only covers $\partial \Lambda$ only once, we have, for any $y \in \Lambda$,
\[
\pm 1
= \mathrm{wind}(\theta|_{\partial D}, y)
= \mathrm{deg}(\theta, \mathring{D}, y)
= \sum_{x \in \theta^{-1}(\{y\})} \mathrm{sign} \det D\theta (x).
\]
Since $\det D\theta$ has a constant sign, there can only be a unique $x \in \mathring{D}$ so that $\theta(x) = y$.
\end{proof}

An illustration of Coons' parameterization method is displayed on Figure~\ref{fig:coons}.
We make the following observations, with $\theta$ denoting Coons' parametrization as defined in~\eqref{eq:coons}.
\begin{itemize}
\item The orientation of the boundary maps $\gamma_1, \gamma_2, \gamma_3, \gamma_4$ is crucial to ensure that $\theta$ is surjective.
\item As soon as the parameter region $\Lambda\subset \R^2$ is compact, simply connected, and such that $\partial \Lambda$ is a Jordan curve, then $\theta$ is automatically continuous and satisfies $\theta(\partial ([-1,1]^2)) = \partial \Lambda$, hence $\theta([-1,1]^2)$ does cover $\Lambda$ according to Proposition~\ref{prop:coons_onto}.
\item The map $\theta$ will be $C^1$ on $[-1,1]^2$ if the boundary maps $\gamma_1,\gamma_2,\gamma_3,\gamma_4$ are $C^1$, and checking that the determinant of $D\theta$ does not vanish on $[-1,1]^2$ can for instance be accomplished using interval arithmetic.
\item Without the assumption that $\theta$ is a local diffeomorphism in $(-1,1)^2$, the image of Coons' parameterization may be strictly larger than $\Lambda$, for instance if $\Lambda$ is very non-convex.
\item Finally, we point out that Coons' parametrization is also relevant when $\Lambda$ is an immersed surface in $\R^{d'}$ with $d' \ge 3$.
Although the map $\theta$ may fail to be surjective, and even to coincide with $\Lambda$, it may nonetheless be useful whenever an approximate representation of $\Lambda$ is sufficient.
\end{itemize}

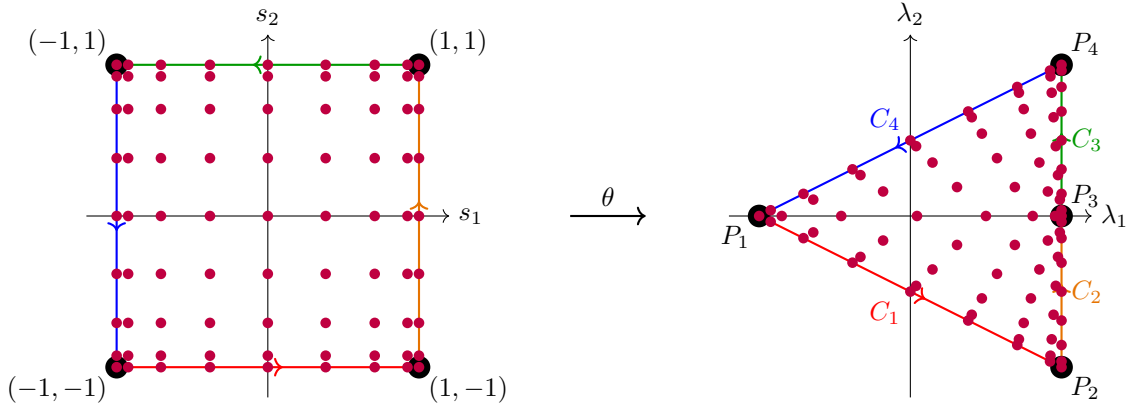
\begin{figure}
    \centering
\begin{tikzpicture}[scale=2]
\begin{scope}
  \draw[->] (-1.2,0) -- (1.2,0) node[right] {$s_1$};
  \draw[->] (0,-1.2) -- (0,1.2) node[above] {$s_2$};

  \draw[thick, red,   arrowmid] (-1,-1)--(1,-1);
  \draw[thick,orange!90!black, arrowmid] (1,-1)--(1,1);
  \draw[thick,green!60!black, arrowmid] (1,1)--(-1,1);
  \draw[thick,blue,  arrowmid] (-1,1)--(-1,-1);

  \fill (-1,-1) circle (0.075);
  \fill (-1, 1) circle (0.075);
  \fill ( 1,-1) circle (0.075);
  \fill ( 1, 1) circle (0.075);

  \node[below left]  at (-1,-1) {$(-1,-1)$};
  \node[below right] at (1,-1)  {$(1,-1)$};
  \node[above left]  at (-1,1)  {$(-1,1)$};
  \node[above right] at (1,1)   {$(1,1)$};
  
\foreach \i in {0,...,8} {
  \foreach \j in {0,...,8} {
    \pgfmathsetmacro{\sone}{-cos(\i*pi/8 r)}
    \pgfmathsetmacro{\stwo}{-cos(\j*pi/8 r)}
    \fill[purple] (\sone,\stwo) circle (0.035);
  }
}
\end{scope}

\draw[thick,->] (2,0) -- (2.5,0) node[midway,above] {$\theta$};










\begin{scope}[shift={(4.25,0)}]
  \draw[->] (-1.2,0) -- (1.2,0) node[right] {$\lambda_1$};
  \draw[->] (0,-1.2) -- (0,1.2) node[above] {$\lambda_2$};

  \coordinate (P1) at (-1,0);
  \coordinate (P2) at (1,-1);
  \coordinate (P3) at (1,0);
  \coordinate (P4) at (1,1);

  \draw[thick,red,   arrowmid] (P1)--(P2)  node[midway, below left] {$C_1$};
  \draw[thick,orange!90!black, arrowmid] (P2)--(P3)  node[midway, right] {$C_2$};
  \draw[thick,green!60!black, arrowmid] (P3)--(P4) node[midway, right] {$C_3$};
  \draw[thick,blue,  arrowmid] (P4)--(P1) node[midway, above left] {$C_4$};

  \foreach \p in {P1,P2,P3,P4} \fill (\p) circle (0.075);

  \node[below left] at (P1) {$P_1$};
  \node[below right] at (P2) {$P_2$};
  \node[above right] at (P3) {$P_3$};
  \node[above right] at (P4) {$P_4$};

  \foreach \i in {0,...,8} {
    \foreach \j in {0,...,8} {

      \pgfmathsetmacro{\sone}{-cos(\i*pi/8 r)}
      \pgfmathsetmacro{\stwo}{-cos(\j*pi/8 r)}
      \pgfmathsetmacro{\alphaone}{(1+\sone)/2}
      \pgfmathsetmacro{\alphatwo}{(1+\stwo)/2}

      \pgfmathsetmacro{\gax}{\sone}
      \pgfmathsetmacro{\gay}{-(1+\sone)/2}

      \pgfmathsetmacro{\gbx}{1}
      \pgfmathsetmacro{\gby}{(-1+\stwo)/2}

      \pgfmathsetmacro{\gcmx}{1}
      \pgfmathsetmacro{\gcmy}{(1-\sone)/2}

      \pgfmathsetmacro{\gdmx}{\stwo}
      \pgfmathsetmacro{\gdmy}{(1+\stwo)/2}

      \pgfmathsetmacro{\tttx}{\gax + \alphatwo*(\gcmx - \gax)}
      \pgfmathsetmacro{\ttty}{\gay + \alphatwo*(\gcmy - \gay)}

      \pgfmathsetmacro{\ttttx}{\gdmx + \alphaone*(\gbx - \gdmx)}
      \pgfmathsetmacro{\tttty}{\gdmy + \alphaone*(\gby - \gdmy)}

      \pgfmathsetmacro{\pcx}{-1 + \alphaone*(1 - (-1)) + \alphatwo*(1 - (-1) + \alphaone*(1 - 1 - 1 - 1))}
      \pgfmathsetmacro{\pcy}{0  + \alphaone*(-1 - 0) + \alphatwo*(1 - 0 + \alphaone*(0 - 1 + 1 + 0))}

      \pgfmathsetmacro{\lx}{\tttx + \ttttx - \pcx}
      \pgfmathsetmacro{\ly}{\ttty + \tttty - \pcy}

      \fill[purple] (\lx,\ly) circle (0.035);
    }
  }
\end{scope}

\end{tikzpicture}
\caption{Chebyshev grid of size $(N_1+1) \times (N_2+1) = (2^3+1) \times (2^3+1) = 9 \times 9$ in the square $[-1,1]^2$ (left) and its image onto the triangle $\Lambda$ under the map~\eqref{eq:coons} where $\gamma_1, \gamma_2, \gamma_3, \gamma_4$ are bijective affine functions satisfying~\eqref{eq:bdy_rel}, with orientation as in the picture (right).}
\label{fig:coons}
\end{figure}



\section{Applications}\label{sec:applications}

As a concluding section, we showcase our rigorous continuation strategy on two examples:
\begin{enumerate}
\item A two-parameter continuation of steady-states for the Cahn--Hilliard equation (Section~\ref{sec:ch}), an elliptic PDE with a bi-Laplacian and a cubic nonlinearity;
\item A pseudo-arclength continuation of steady-states for the Shigesada--Kawasaki--Teramoto system (Section~\ref{sec:skt}), an elliptic PDE system with nonlinear diffusion and quadratic reaction terms.
\end{enumerate}

In both cases, the PDEs are set on the spatial domain $(0, 1)$, together with Neumann boundary conditions.
Thus, a steady-state $v=v(y)$ for $y\in[0,1]$ is sought as a cosine series
\begin{equation}
v(y) = \sum_{k \in \mathbb{Z}} v_{|k|} \cos ( \pi k y ).
\end{equation}
The analytic regularity of the solutions suggests that the cosine series coefficients are, a fortiori, elements of the sequence space $\cX = \ell^1_\chi(\R)$, as defined in~\eqref{eq:weighted_ell1}, and with a weight sequence $\chi$ satisfying~\ref{A2}.
Note that $\cX$ inherits a multiplication operation from the product of cosine series, namely $\odot : \cX \times \cX \to \cX$ given, for any $\phi, \psi \in \cX$, by
\begin{equation}\label{eq:conv_cos}
(\phi \odot \psi)_k = \sum_{k' \in \Z} \phi_{|k - k'|} \psi_{|k'|}, \qquad k \in \N.
\end{equation}
Observe that the formula is identical to the multiplication $*$ of Chebyshev series~\eqref{eq:conv}, which is to be expected from the identity~\eqref{eq:cheb_identity}.
The Banach space for the continuation is then taken to be $\ell^{1,\otimes d}_\omega(\cX)$, as defined in~\eqref{eq:ell1_multi}.
Hence, the solution manifold of steady-state is expressed as, for $y \in [0, 1]$ and $\lambda = (\lambda_1, \dots, \lambda_d) \in [-1, 1]^d$,
\begin{equation}
[v(\lambda)](y) = \sum_{n = (n_1, \dots, n_d) \in \Z^d, k \in \Z} v_{|n_1|, \dots, |n_d|, |k|} \cos ( \pi k y ) T_{|n_1|}(\lambda_1) \cdots T_{|n_d|}(\lambda_d).
\end{equation}
%





Unlike the toy examples studied so far in Section~\ref{sec:cubic_root} and Section~\ref{sec:sqrt_example}, the problems for fixed parameters are now themselves infinite dimensional.
To address it, an important operator to introduce is the \emph{truncation operator} $\Pi_{\le K} : \ell^{1,\otimes d}_\omega(\cX) \to \ell^{1,\otimes d}_\omega(\cX)$ given component-wise by
\begin{equation}
(\Pi_{\le K} \psi)_{n_1, \dots, n_d, k} \bydef
\begin{cases}
\psi_{n_1, \dots, n_d, k}, & k \le K, \\
0, & k > K,
\end{cases} \quad n = (n_1, \dots, n_d) \in \N^d, \qquad \text{for all } \psi \in \ell^{1,\otimes d}_\omega(\cX).
\end{equation}
Its complement is the \emph{tail operator} $\Pi_{> K} \bydef I - \Pi_{\le K}$.
Finally, in order to describe the range of the zero-finding problem $F$, let us introduce $\cY = \ell^1_\zeta(\R)$, where $\zeta_k = \chi_k (1+k)^{-2}$ for all $k\in\N$. The Banach space $\cY$ is not necessarily a Banach algebra (the weight sequence $\zeta$ may not satisfy~\eqref{A2}), but this is not an issue. As shown by the assumptions of Theorem~\ref{th:NK}, all the estimates happen in $\cX$, and we merely need to check that $A$ maps $\cY$ back into $\cX$ so that $I-AF$ is a suitable fixed-point map, but this is usually evident from the definition of $A$.

For both of the examples to come, we emphasize that computer-assisted proofs for fixed parameters can already be found in the literature. Our main purpose here is to illustrate that going from such proofs to validating entire curves or surfaces of solutions is straightforward with the approach proposed in this paper. We also point out that this approach already proved very useful for a broad range of problems, including the resolution of a conjecture in celestial mechanics~\cite{CalGarHenLesMir24}, the rigorous study of large deviation estimates for finite time Lyapunov exponents~\cite{BleBluBreEng25}, a rigorous proof of homoclinic snaking~\cite{BerDucLes25}, or the resolution of a conjecture about overhanging water-waves~\cite{CadHaz26}.

\subsection{Two-parameter continuation of steady-states in the Cahn--Hilliard equation}
\label{sec:ch}

Consider the Cahn--Hilliard equation
\begin{equation}\label{eq:ch}
\begin{cases}
\partial_t v = - \Delta (\epsilon^2 \Delta v + v - v^3), & y \in (0, 1),\\
\partial_y v = \partial_y^3 v = 0, & y = 0, 1.
\end{cases}
\end{equation}
Noting that the total mass $\sigma \bydef \int_0^1 v(t, y) \, \mathrm{d} y$ is conserved, and
reorganizing the equations, a steady-state $v = v(y)$ satisfies
\begin{equation}
\begin{cases}
\epsilon^2 \Delta v + v - v^3 = c, & y\in(0, 1),\\
\partial_y v = 0, & y = 0, 1.
\end{cases}
\end{equation}
where $c = \int_0^1 v(y) - v(y)^3 \, \mathrm{d} y$.
In particular, $\sigma = c + \int_0^1 v(y)^3 \, \mathrm{d} y$, and one can equivalently vary $c$ or $\sigma$.

To our knowledge, the only work performing rigorous two-parameter continuation on the steady-state Cahn--Hilliard equation is~\cite{GamLesPug16}, which is also the article that introduced the multi-dimensional rigorous continuation used thus far in the literature; we use it as a main reference point to draw comparison with our method.

\subsubsection*{Step 1: Writing the continuation problem}

Writing $\beta = 1/\epsilon^2$, the two-parameter continuation problem has the form~\eqref{eq:Fxl} with the map $F : \cX \times \Lambda \to \cY$ (where $\Lambda$ is prescribed below) given by 
\begin{equation}\label{eq:steadystateCH}
F (v, \lambda) = \Delta v + \beta (v - v^{\odot 3} - c), \qquad \text{for all } v \in \cX, \lambda = (c, \beta) \in \Lambda.
\end{equation}
In the above expression, we emphasized the discretization of the function space by using the symbol $\odot$ corresponding to the product of cosine series, see~\eqref{eq:conv_cos}, and where $\Delta$ acts on a sequence $\psi \in \cX$ as
\begin{equation}
(\Delta \psi)_k = -\pi^2 k^2 \psi_k, \qquad k \in \N.
\end{equation}
Let
\begin{equation}
g(c) \bydef \frac{79}{2} + 150 c^2, \qquad c_{\max} \bydef \sqrt{\frac{41 - 79/2}{150}} = \frac{1}{10}.
\end{equation}
As our goal is to reproduce the two-parameter validated continuation conducted in~\cite{GamLesPug16}, we consider the parameter set given by
%
\begin{equation}\label{eq:Ushape}
\Lambda = \left\{(c, \beta) \in \R^2 \,: \, g(c) \le \beta \le 41, \, |c| \le c_{\max} \right\}.
\end{equation}
This parameter set is represented on Figure~\ref{fig:coons_parabola}.

\begin{remark}
As explained in~\cite{GamLesPug16}, there is a curve of fold bifurcations in the $(c, \beta)$-plane, near the U-shaped part of the boundary of $\Lambda$. We note that our parameter set $\Lambda$ is not \emph{exactly} the same as the one used in~\cite{GamLesPug16}, as the latter was obtained using a simplicial decomposition and therefore does not have such a simple expression as ours.
Though, both are very similar and, in particular, are roughly the same distance away from the bifurcation curve. 
\end{remark}

We consider two different maps from $[-1, 1]^2$ onto $\Lambda$ using Coons' parameterization method discussed in Section~\ref{sec:multi_dim}.
The first one is produced by considering the bijective set of boundary curves:
\begin{equation}\label{eq:parabolacurves1}
\begin{aligned}
\gamma_1^{(no \; pinch)}(s) &= \Bigg(\frac{s-3}{4} c_{\max} , g\Big(\frac{s-3}{4} c_{\max} \Big)\Bigg), \qquad
\gamma_2^{(no \; pinch)}(s) = \Bigg(\frac{s}{2} c_{\max}, g\Big(\frac{s}{2} c_{\max}\Big)\Bigg), \\
\gamma_3^{(no \; pinch)}(s) &= \Bigg(\frac{s+3}{4} c_{\max}, g\Big(\frac{s+3}{4} c_{\max}\Big)\Bigg), \qquad
\gamma_4^{(no \; pinch)}(s) = \Bigg(-s c_{\max}, 41\Bigg).
\end{aligned}
\end{equation}
The second map is generated by taking an arguably simpler parameterization, but for which, in contrast, some of the boundary curves are non-injective:
\begin{equation}\label{eq:parabolacurves2}
\begin{aligned}
\gamma_1^{(pinch)}(s) &= \big(-c_{\max} , 41\big), \qquad
\gamma_2^{(pinch)}(s) = \big(s c_{\max}, g(s c_{\max})\big), \\
\gamma_3^{(pinch)}(s) &= \big(c_{\max} , 41\big),\phantom{-a} \qquad
\gamma_4^{(pinch)}(s) = \big(-s c_{\max}, 41\big).
\end{aligned}
\end{equation}
Each set of boundary curves yields a Coon parameterization $\theta:[-1,1]\to\Lambda$ as described in~\eqref{eq:coons}.
Figure~\ref{fig:coons_parabola} shows how these two parameterizations maps Chebyshev nodes over $\Lambda$.
Specifically, the parameterization generated by~\eqref{eq:parabolacurves1} (middle figure) distributes much more uniformly the Chebyshev nodes, compared to the parameterization generated by~\eqref{eq:parabolacurves2} (right figure) where an unbalanced proportion of nodes accumulate near the two points $P_1 = P_4$ and $P_2 = P_3$ where injectivity fails.
We will see below to which extend these differences affect the rigorous continuation procedure.

\begin{figure}[ht!]
\centering
\includegraphics{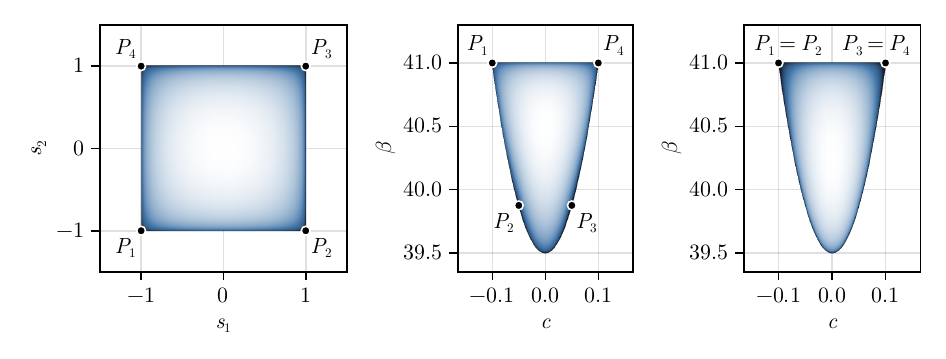}
\caption{Density plot of the Chebyshev grid in the square $[-1,1]^2$ (left) and its image onto the parameter set $\Lambda$ from~\eqref{eq:Ushape}, using the set of boundary curves with no pinch~\eqref{eq:parabolacurves1} (middle) and using the set of boundary curves with pinch~\eqref{eq:parabolacurves2} (right).}
\label{fig:coons_parabola}
\end{figure}

Now that we have a parameterization of $\Lambda$, we have everything to use the framework for the contraction argument discussed in Section~\ref{sec:framework}.
Introduce $u = v \circ \theta$, where $\theta(s) = (\theta_1(s), \theta_2(s)) = (c, \beta) \in \Lambda$, yielding the continuation problem
\begin{equation*}
0 = \tilde{F} (u, s) = \Delta u + \theta_2(s) (u - u^{\odot 3} - \theta_1(s)), \qquad \text{for all } u \in \cX, s \in [-1, 1]^2.
\end{equation*}
The superposition operator $\tilde\cF : \ell^{1, \otimes2}_\omega(\cX) \to \ell^{1, \otimes2}_\omega(\cY)$ associated with $\tilde{F}$ is the mapping 
\begin{equation*}
\tilde\cF (u) = \mathcal{D}^2 u + \theta_2 * (u - u^{* 3} -\theta_1), \qquad \text{for all } u \in \ell^{1, \otimes2}_\omega(\cX),
\end{equation*}
where $\theta = (\theta_1, \theta_2) : [-1, 1]^2 \to \Lambda$ such that $\theta_1, \theta_2 \in \ell^{1, \otimes2}_\omega(\cX)$ are functions of the parameter $s = (s_1, s_2) \in [-1, 1]^2$ and take the form
\[
\theta_j(s) = \sum_{n = (n_1, n_2) \in \Z^2} (\theta_j)_{|n_1|, |n_2|} T_{|n_1|}(s_1) T_{|n_2|}(s_2), \qquad j=1,2.
\]
Note also that $\mathcal{D}^2 = (\Delta, 0, 0, \dots)$ acts on an element $\psi \in \ell^{1, \otimes2}_\omega(\cX)$ via
\begin{equation*}
(\mathcal{D}^2 \psi)_{n_1, n_2, k} = -\pi^2 k^2 \psi_{n_1, n_2, k}, \qquad n = (n_1, n_2) \in \N^2, k \in \N.
\end{equation*}
%

\subsubsection*{Step 2: Computing $\bu$ and $\cA$}

For $s^{(j)} = (s_1^{(j_1)}, s_2^{(j_2)})$, we solve approximately the truncated problem $ \Pi_{\le K} \tF(\Pi_{\le K} u^{(j)}, s^{(j)}) \approx 0$.
Then, for each cosine coefficient $u^{(j)}_k$, we perform a two dimensional inverse DCT to find the Chebyshev interpolation polynomial (or, an approximation of it) $\bu_k(s)$; by construction,
\[
\bu(s) = \begin{pmatrix}
\bu_0(s) \\
\vdots \\
\bu_{K}(s) \\
0 \\
\vdots
\end{pmatrix}, \qquad \bu_k(s) = \bu_{k,0} + 2 \sum_{n = 1}^N \bu_{k,n} T_n (s), \quad k = 0, \dots, K.
\]
Next, we explain how to obtain a suitable $\cA$.
For fixed parameters, the construction is standard in the field of computer-assisted proofs based on Theorem~\ref{th:NK}, and based on using a (relatively) compact perturbation of the leading differential operator $\Delta$ (see Appendix~\ref{app:ch}).
We proceed exactly the same way here, at each Chebyshev nodes, and then interpolate.

To start, note that $D_u \tF(u, s) \psi = \Delta \psi + \theta_2(s) (1 - 3u^{\odot 2}) \odot \psi$.
We thus pick an approximate inverse to $D_u \tF(u^{(j)}, s^{(j)})$ given by $A^{(j)} = A_{\le K}^{(j)} \Pi_{\le K} \oplus \Delta_0^{-1} \Pi_{> K}$, where $\Delta_0 \bydef \Delta + \Pi_{\le 0}$ and $A_{\le K}^{(j)}$ is a numerically computed inverse of $\Pi_{\le K}D_u \tF(u^{(j)}, s^{(j)})\Pi_{\le K}$.
We can represent the approximate inverse more visually:
\begin{equation*}
A^{(j)}
=
\begin{pmatrix}
A_{\le K}^{(j)} \\
& \Delta_0^{-1}\Pi_{> K}
\end{pmatrix}
=
\begin{pmatrix}
A_{\le K}^{(j)} \\
& \frac{1}{-\pi^2 (K+1)^2} \\
& & \ddots \\
& & & \frac{1}{-\pi^2 k^2} \\
& & & & \ddots
\end{pmatrix}.
\end{equation*}
The approximate inverse $A\in\ell^{1, \otimes2}(\B(\cY,\cX))$ is obtained by considering the interpolation polynomial of the $A^{(j)}$.
Then, we pick $\cA = \cM(A)$, with
\[
A(s) =
\begin{pmatrix}
A_{\le K}(s) \\
& \Delta_0^{-1}\Pi_{> K}
\end{pmatrix},
\]
where $A_{\le K}(s)$ is a $K$-by-$K$ matrix valued, order $(N,N)$ Chebyshev polynomial in $s$.
In particular, the tail block $\Delta_0^{-1}\Pi_{> K}$ is constant along the solution family as it does not depend on $s$. Also note that, for each $s$, $A(s)$ maps $\cY$ into $\cX$, therefore $\cA$ does maps $\ell^{1, \otimes2}_\omega(\cY)$ into $\ell^{1, \otimes2}_\omega(\cX)$. Finally, $A(s)$ is a Fredholm operator of index $0$ for each $s$, hence we will get for free that $\cA$ is injective if $Z_1<1$ (see Nota Bene~\ref{rem:Ainj}).

\subsubsection*{Step 3: Writing the bounds}

We are now ready to derive the bounds $Y$, $Z_1$ and $Z_2$ required for applying Theorem~\ref{th:NK}, for the map $\tilde\cF:\ell^{1, \otimes2}_\omega(\cX)\to\ell^{1, \otimes2}_\omega(\cY)$ defined in Step 1.
Readers unfamiliar with such a posteriori validation methods may first want to get acquainted with these techniques in the simpler setting of fixed parameter values.
Appendix~\ref{app:ch} is written precisely for this purpose. 

We stress once again (see Sections~\ref{sec:take_home_message} and~\ref{sec:punct2unif}), that the bounds for the continuation are obtained through a lift of the bounds at fixed parameter values, which simply amounts to a change of product and norm.
This is an important appeal of our method, namely that once the problem is understood at a fixed parameter value, continuation requires, so to speak, no extra estimates.

Based on the pointwise estimates derived in Appendix~\ref{app:ch} and on the discussions of Section~\ref{sec:punct2unif}, we take
\begin{equation*}
    Y = \| \Pi_{3 K} \cA \left(\Delta u + \theta_2 * (u - u^{* 3} - \theta_1)\right) \|_{\ell^{1, \otimes2}_\omega(\cX)},
\end{equation*}
\begin{equation*}
    Z_1 = \max \left(
\| \Pi_{3 K} - \Pi_{5 K} \cA D \tilde\cF(\bu) \Pi_{3 K} \|_{\B\left(\ell^{1, \otimes2}_\omega(\cX),\, \ell^{1, \otimes2}_\omega(\cX)\right)},\ \frac{1}{\pi^2 (K+1)^2} \| \theta_2 * (e_{0,0} - 3 \bu^{* 2}) \|_{\ell^{1, \otimes2}_\omega(\cX))}
\right),
\end{equation*}
where $e_{0,0}$ is the element of $\ell^{1, \otimes2}_\omega(\cX)$ corresponding to the constant function $1$ (i.e., $(e_{0,0})_{0,0}=1$ and $(e_{0,0})_{n,k}=0$ otherwise, for $n\in\N^2$ and $k\in\N$), and
\begin{equation*}
    Z_2 = 3\|\theta_2\|_{\ell^{1,\otimes2}_\omega(\cX)}\|\cA\|_{\B\left(\ell^{1, \otimes2}_\omega(\cX),\, \ell^{1, \otimes2}_\omega(\cX)\right)} (2\|\bu\|_{\ell^{1,\otimes2}_\omega(\cX)} + \rstar),
\end{equation*}
where
\begin{equation*}
\| \cA \|_{\B\left(\ell^{1,\otimes2}_\omega(\cX),\ell^{1,\otimes2}_\omega(\cX)\right)} = \max\left( \|\cA_{\le K}\|_{\B\left(\ell^{1,\otimes2}_\omega(\cX),\ell^{1,\otimes2}_\omega(\cX)\right)}, \frac{1}{\pi^2 (K+1)^2} \right),
\end{equation*}
with $\cA_{\leq K} = \cM(A_{\leq K})$.
All the finite calculations required to evaluate these estimates are performed following the discussion in Section~\ref{sec:punct2unif}.

\begin{figure}[ht!]
\centering
\includegraphics[scale=0.6]{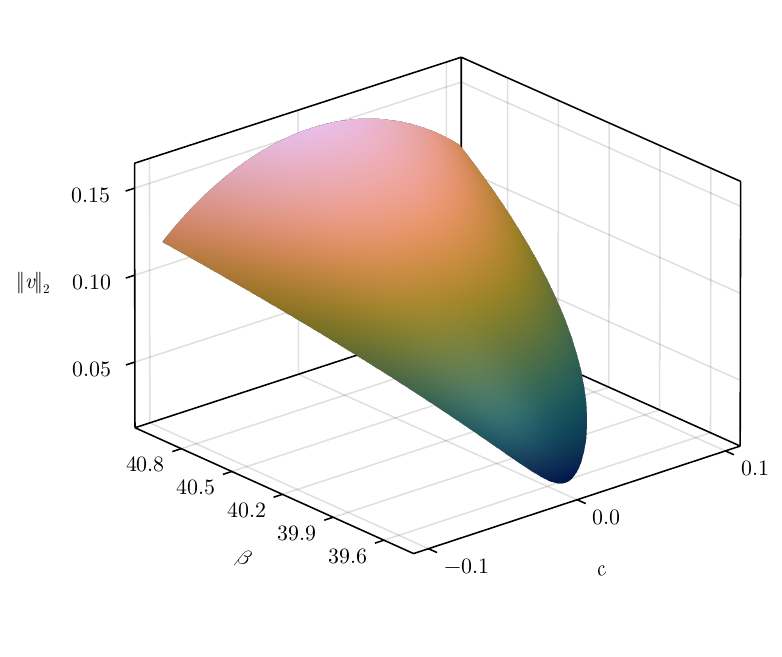}
\caption{A validated manifold of approximate steady states $\bu$ of the Cahn--Hilliard equation~\ref{eq:ch}. This is the approximation obtained using the parameterization $\theta$ of $\Lambda$ without pinch given by~\eqref{eq:parabolacurves1}, but the other one looks indistinguishable at this scale.}
\label{fig:ch_manifold}
\end{figure}

\paragraph{Results of the code.}
We choose an order $(N_1,N_2) = (16, 16)$ in Chebyshev and an order $K = 20$ in Fourier for approximating the steady-states.
The weight sequences are chosen trivial $\chi_k \equiv 1$ and $\omega_n \equiv 1$.
For both parameterizations of $\Lambda$ (with pinch~\eqref{eq:parabolacurves1} and without~\eqref{eq:parabolacurves2}) the execution of the code in \cite{CODE} produces a two-dimensional manifold of approximate solutions $\bu$, represented on Figure~\ref{fig:ch_manifold}, and certifies the following bounds
\begin{alignat*}{3}
Y^{(no \; pinch)} &= 1.23433 \times 10^{-5}, \qquad
&&Y^{(pinch)} &&= 3.6778 \times 10^{-5}, \\
Z_1^{(no \; pinch)} &= 3.24132 \times 10^{-2}, \qquad
&&Z_1^{(pinch)} &&= 9.65821 \times 10^{-2}, \\
Z_2^{(no \; pinch)} &= 540.4, \qquad
&&Z_2^{(pinch)} &&= 489.731, 
\end{alignat*}
with
%
\begin{equation*}
\rstar^{(no \; pinch)} = 1.23432 \times 10^{-4}, \qquad
\rstar^{(pinch)} = 3.67779 \times 10^{-4}, \qquad
\end{equation*}
and a valid error bound for the entire manifold is, in each case,
\begin{equation*}
r^{(no \; pinch)} = 1.28025 \times 10^{-5}, \qquad r^{(pinch)} = 4.11693 \times 10^{-5}.
\end{equation*}
It is remarkable that despite the striking difference in the density of the Chebyshev nodes, as shown on Figure~\ref{fig:coons_parabola}, and the resulting different maps $\theta:s\mapsto(c,\beta)$ shown on Figure~\ref{fig:ch_param}, the final error bound is of the same order in both cases, and only slightly worse for the pinched parameterization of the boundary curves.
Admittedly, this example is somewhat simple (though non trivial...) as a few Fourier modes sufficed to represent accurately each steady-state on the solution manifold.
For a more difficult problem, however, it might be more critical to parameterize the boundary with some amount of care.

We also point out the clear computational and accuracy advantages of our strategy over the approach of~\cite{GamLesPug16}.
In that work, the solution manifold was represented by a triangulation, and a separate validation was carried out on each simplex.
It was reported that this piecewise quadratic approximation of the manifold required approximately 39 hours, used 500 processors, with a total CPU time of about 680 days.
By contrast, the approach developed in the present article validates essentially the same solution manifold (up to a different boundary) in roughly 2 minutes on a laptop equipped with a single M5 chip and 48 GB of RAM.

Moreover, while~\cite{GamLesPug16} does not report rigorous error bounds, private communication with the authors confirmed that the error bound obtained here (of order $10^{-5}$, and produced through a single validation) is smaller than each of their individual proofs over a simplex, which is no better than order $10^{-3}$.

\begin{figure}[ht!]
\centering
\includegraphics[scale=0.6]{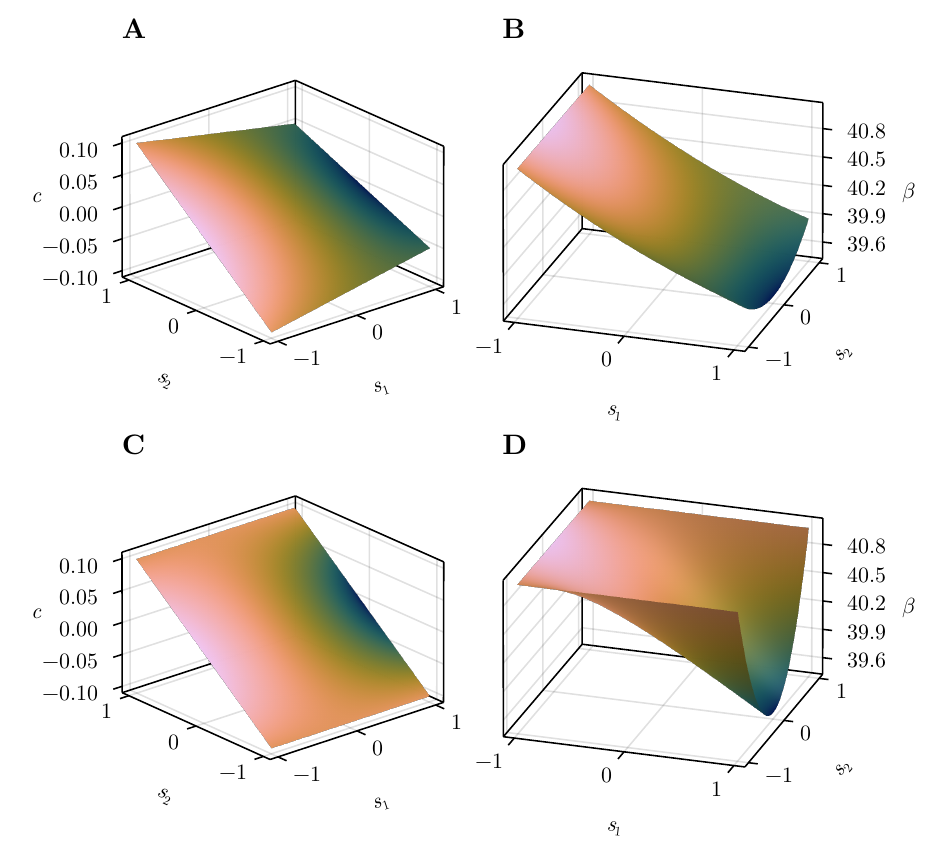}
\caption{The maps $c$ and $\beta$ forming the parameterization $\theta=(c,\beta):[-1,1]^2\to\Lambda$, obtained using the set of boundary curves with no pinch~\eqref{eq:parabolacurves1} (\textbf{A} and \textbf{B}) and with pinch~\eqref{eq:parabolacurves2} (\textbf{C} and \textbf{D}).}
\label{fig:ch_param}
\end{figure}

\subsection{Pseudo-arclength continuation of steady-states in the Shigesada--Kawasaki--Teramoto system}
\label{sec:skt}

The so-called SKT model is a system of PDEs introduced in the seminal work~\cite{ShiKawTer79}, and is used to model the evolution of two competing species. It writes
\begin{equation}\label{eq:skt}
\begin{cases}
\partial_t v_1 = \Delta ((d_1 + d_{11} v_1 + d_{12} v_2)v_1) + (r_1 - a_1 v_1 - b_1 v_2)v_1, \\
\partial_t v_2 = \Delta ((d_2 + d_{21} v_1 + d_{22} v_2)v_2) + (r_2 - b_2 v_1 - a_2 v_2)v_2, \\
\frac{\partial v_1}{\partial n} = \frac{\partial v_2}{\partial n} = 0,
\end{cases}
\end{equation}
where $v_1 = v_1(t,y)$ and $v_2 = v_2(t,y)$ are the population densities of two interacting species. We consider here the one-dimensional case, and fix the domain to be $(0, 1)$.

The reactions terms are standard Lotka-Volterra terms with signs indicating intra-specific and inter-specific competition (all the parameters of the model are nonnegative). The key feature of this model is the presence of nonlinear diffusion terms: the diffusion rate of each species depends on the local density of both species. In particular, the cross-diffusion terms can give rise to a repulsive effect which leads to \emph{spatial segregation}: the two species co-exist but mostly concentrate in different regions of the domain $\Omega$. Mathematically, this corresponds to the existence of non-homogeneous steady-states of~\eqref{eq:skt} exhibiting this type of patterns. 

The existence of nontrivial steady-states of~\eqref{eq:skt} has been studied extensively, through a wide variety of techniques such as bifurcation theory, singular perturbation theory or fixed point index theory, see e.g.~\cite{MimKaw80,MimNisTesTsu84,LouNi96,RyuAhn03}, which provide a qualitative understanding on the conditions required on the many parameters of~\eqref{eq:skt} for non-homogeneous steady-states to exist. However, numerical studies like~\cite{IidMimNim06,IzuMim08,BreKueSor21} show that the steady-states of~\eqref{eq:skt} can be very diverse, already in the one dimensional case, and that many of them can co-exists for a given set of parameter values, leading to intricate bifurcation diagrams as shown for instance on Figure~\ref{fig:skt_bif}. 
Proving the existence, and characterizing the shape, of all these different steady-states seems very hard by purely analytic means, but computer-assisted techniques provide a valuable tool to attack such questions. Indeed, the paper~\cite{Bre22} contains a general computer-assisted methodology based on Theorem~\ref{th:NK}, allowing to study steady-states of~\eqref{eq:skt} for fixed parameters. It should be noted that rigorous validation of branches of solutions was already conducted in~\cite{BreLesVan13}, but only for a system approximating~\eqref{eq:skt}. Moreover, this was done using the piece-wise linear approach described in the introduction, which meant that validating a single piece of branch (between two successive bifurcations) required splitting the branch into thousands of small sub-pieces, as well as the derivation of extra estimates for the continuation. We show here that the approach proposed in this paper allows for the validation of a whole branch \emph{in one go} (or at least, using only a couple of sub-pieces), and with no extra estimates compared to~\cite{Bre22}. As in Figure~\ref{fig:skt_bif}, we now take $\lambda = d_1 = d_2$ as the bifurcation parameter, while all other parameters are fixed.



Steady-state solutions $v = (v_1, v_2)$ of the SKT system satisfy the time-independent version of \eqref{eq:skt}, that is
\begin{equation}
0 = F(v, \lambda) = \Delta \Phi(v, \lambda) + R(v),
\end{equation}
where
\begin{equation}
\Phi(v, \lambda) \bydef \begin{pmatrix} (\lambda + d_{11} v_1 + d_{12} v_2)v_1 \\ (\lambda + d_{21} v_1 + d_{22} v_2)v_2 \end{pmatrix}, \qquad
R(v) \bydef \begin{pmatrix} (r_1 - a_1 v_1 - b_1 v_2)v_1 \\ (r_2 - b_2 v_1 - a_2 v_2)v_2 \end{pmatrix}.
\end{equation}
The bifurcation diagram of Figure~\ref{fig:skt_bif} indicates the occurrence of fold bifurcations along several solution branches.
To go through these singular points, we use the pseudo-arclength method as detailed in Section~\ref{sec:pseudo_arclength} so that we consider
\[
F_\textnormal{pa}(u, s) = \begin{pmatrix} F(v, \lambda) \\ \langle *, * \rangle \end{pmatrix},
\]
where $u = (v, \lambda)$ are the unknown variables, and $s$ is the pseudo-arclength parameter. 

Similarly to what we did for the Cahn--Hilliard example, for fixed parameters we look for steady-states in the space $\cX = \ell^1_\chi(\R) \times \ell^1_\chi(\R)$ (one copy of $\ell^1_\chi(\R)$ for each component of $v$), endowed with the norm
\[
\| v \|_\cX = \| v_1 \|_{\ell^1_\chi(\R)} + \| v_2 \|_{\ell^1_\chi(\R)},
\]
and take $\cY = \ell^1_\zeta(\R) \times \ell^1_\zeta(\R)$.
The extended zero-finding problem $F_\pa$ is set between $\cX_\pa=\cX\times \R$ and $\cY_\pa=\cY\times \R$. Following the general setting developed in Section~\ref{sec:framework}, and in particular Remark~\ref{rem:systems}, we then study the continuation problem by considering the superposition operator $\cF_\pa$ associated to $F_\pa$, on the space $\ell^1_\omega(\ell^1_\chi(\R))\times \ell^1_\omega(\ell^1_\chi(\R))\times \ell^1_\omega(\R) \cong \ell^1_\omega(\cX_\pa)$, for which we take
\begin{equation*}
    \left\Vert u\right\Vert_{\ell^1_\omega(\ell^1_\chi(\R))\times \ell^1_\omega(\ell^1_\chi(\R))\times \ell^1_\omega(\R)} = \left\Vert v_1\right\Vert_{\ell^1_\omega(\ell^1_\chi(\R))} + \left\Vert v_2\right\Vert_{\ell^1_\omega(\ell^1_\chi(\R))} + \left\Vert \lambda\right\Vert_{\ell^1_\omega(\R)}, 
\end{equation*}
for all $u=(v,\lambda) = (v_1,v_2,\lambda) \in \ell^1_\omega(\ell^1_\chi(\R))\times \ell^1_\omega(\ell^1_\chi(\R))\times \ell^1_\omega(\R)$. In order to shorten subsequent formula, we introduce $\cS = \ell^1_\omega(\ell^1_\chi(\R))\times \ell^1_\omega(\ell^1_\chi(\R))$ and $\cS_\pa = \cS\times \ell^1_\omega(\R)$. 

Compared to the Cahn--Hilliard example of Section~\ref{sec:ch}, the main difference here is that the presence nonlinear diffusion term $\Phi$ requires a more involved approximate inverse $A$. We follow the construction introduced in~\cite{Bre22}, and then use the same bounds, with minor modifications to account for the pseudo-arclength continuation setting. Splitting the Fréchet derivative of $F_\textnormal{pa}$ at $\bar{u} = (\bv, \bar{\lambda})$ as
\begin{equation}
D_u F_\text{pa}(\bar{u}(s), s) = 
\begin{pmatrix} \Delta D_v \Phi(\bv(s), \bar{\lambda}(s)) & 0 \\ 0 & 0 \end{pmatrix} + 
\begin{pmatrix} D R(\bv(s)) & \Delta D_\lambda \Phi(\bv(s), \bar{\lambda}(s)) \\ * & * \end{pmatrix},
\end{equation}
we define, for some integer $K_A \ge 1$,
\begin{equation}
A(s) = A_{\le K_A}(s) \Pi_{\le K_A} + \begin{pmatrix} (W(s) - \Pi_{\le K_A}  W(s) \Pi_{\le K_A})\Delta_0^{-1}  & 0 \\ 0 & 0 \end{pmatrix},
\end{equation}
where $W(s) \approx D_v \Phi(\bv(s), \bar{\lambda}(s))^{-1}$ is constructed by inverting approximately in grid space the collection of 2-by-2 matrix $D_v \Phi([\bv(s)](y), \bar{\lambda}(s))$ and going back to coefficient space. Here again, for fixed $s$, $A(s)$ is a Fredholm operator of index $0$ (see~\cite{Bre22} for details).

\begin{remark}
The two truncation orders $K_A$ and $K$ play distinct roles and can be chosen independently.
The former controls the order to which $A$ accurately approximates $D\cF_\textnormal{pa}$ (controlled by the $Z_1$ bound), while the latter controls the accuracy with which $\bar{u}$ approximates the solution (controlled by the $Y$ bound).
Often $K_A$ is chosen equal to $K$ for simplicity (as we did in the Cahn--Hilliard example of Section~\ref{sec:ch}), however it will be helpful here to take a larger value $K_A>K$.
\end{remark}


As in the Cahn--Hilliard example in Section~\ref{sec:ch}, the $Y$ bound simply taken to be $Y=\Vert \cA \cF_\pa(\bu) \Vert_{\cS_\pa}$, as this is directly computable. We show below the formulas for $Z_1$ and $Z_2$, and again refer to~\cite{Bre22} for more details on their derivation in the pointwise case. Based on those, we can take
\begin{align*}
    Z_1 = \max & \Big( \| \Pi_{\leq K+K_A} - \Pi_{\leq 2K+K_A} \cA  D_u \cF_\pa(\bu) \Pi_{\leq K+K_A}\|_{\B\left(\cS_\pa,\cS_\pa\right)}  , \\
    &\ \ \| \mathcal{I} - \cM(W) \cM(D_v \Phi (\bv, \bar{\lambda})) \|_{\B\left(\cS,\cS\right)} + \frac{1}{\pi^2(K_A+1)^2} \| \cM(W) \|_{\B\left(\cS,\cS\right)} \| \cM(DR(\bv)) \|_{\B\left(\cS,\cS\right)} \Big),
\end{align*}
where the Fr\'echet derivatives of $\Phi$ and $R$ read
\begin{align*}
    D_v\Phi(v,\lambda) &= \begin{pmatrix} \lambda + 2d_{11} v_1 + d_{12}v_2 & d_{12} v_1 \\ d_{21} v_2 & \lambda + d_{21}v_1 + 2d_{22}v_2 \end{pmatrix}, \\
DR(v) &= \begin{pmatrix} r_1 - 2a_1 v_1 - b_1 v_2 & -b_1 v_1 \\ -b_2 v_2 & r_2 - b_2 v_1 - 2a_2 v_2 \end{pmatrix}.
\end{align*}
In practice, we in fact replace the $\B\left(\cS,\cS\right)$ norms appearing here by the simpler upper bounds provided by~\eqref{eq:normAgeneral_syst}. In particular, we never actually build the various multiplication operators involved in those estimates. 

The remaining Lipschitz estimate $Z_2$ is obtained by first noting that since the nonlinearities are quadratic so we may take $\rstar = \infty$. Next, denoting by $\cBL(\cX\times\cY,\cZ)$ the space of bounded bilinear operators from $\cX\times\cY$ into $\cZ$, we have that
\begin{align*}
\| \mathcal{A} D^2 \cF_\text{pa}(u) \|_{\cBL(\cS_\pa\times\cS_\pa,\cS_\pa)}
&\le \left\Vert \cA \begin{pmatrix}\Delta & 0 \\ 0 & 0 \end{pmatrix} \right\Vert_{\B(\cS_\pa,\cS_\pa)} \| \cM(D_v^2 \Phi(v, \lambda)) \|_{\cBL(\cS\times\cS,\cS)} \\
&\quad +
\| \cA \|_{\B(\cS_\pa,\cS_\pa)} \| \cM(D^2 R (v)) \|_{\cBL(\cS\times\cS,\cS)},
\end{align*}
where
\begin{align*}
\|\cM(D_v^2 \Phi(v, \lambda)) \|_{\cBL(\cS\times\cS,\cS)} &\le \max\Big( \max(|2d_{11}|, |d_{12}|) + |d_{21}|,  |d_{12}| + \max(|d_{21}|, |2d_{22}|) \Big) = C_1, \\
\| \cM(D^2 R (v)) \|_{\cBL(\cS\times\cS,\cS)} &\le \max\Big( \max(|2a_1|, |b_1|) + |b_2|, |b_1| + \max(|b_2|, |2a_2|) \Big) = C_2.
\end{align*}
and therefore we can take 
\begin{equation*}
    Z_2 = \left\Vert \cA \begin{pmatrix}\Delta & 0 \\ 0 & 0 \end{pmatrix} \right\Vert_{\B(\cS_\pa,\cS_\pa)} C_1 + \| \cA \|_{\B(\cS_\pa,\cS_\pa)} C_2.
\end{equation*}  
Once again, in practice we replace the remaining operator norms by upper bounds involving a finite part and then a tail estimate obtained here using~\eqref{eq:normAgeneral_syst}:
\begin{align*}
    \| \cA \|_{\B(\cS_\pa,\cS_\pa)}\leq \max \left(\|  \cA\Pi_{\leq K_A} \|_{\B(\cS_\pa,\cS_\pa)},\, \frac{1}{\pi^2(K_A+1)^2} \| \cM(W) \|_{\B\left(\cS,\cS\right)} \right),
\end{align*}
and similarly
\begin{align*}
  \left\Vert \cA \begin{pmatrix}\Delta & 0 \\ 0 & 0 \end{pmatrix} \right\Vert_{\B(\cS_\pa,\cS_\pa)} \leq   
  \max \left(\left\Vert \cA\Pi_{\leq K_A} \begin{pmatrix}\Delta & 0 \\ 0 & 0 \end{pmatrix} \right\Vert_{\B(\cS_\pa,\cS_\pa)},\, \| \cM(W) \|_{\B\left(\cS,\cS\right)} \right).
\end{align*}

Finally, let us point out that sharper (but more cumbersome to write down) estimates could be obtained for the final term appearing in $Z_1$ as well as for the $Z_2$ bound, and we refer again to~\cite{Bre22} for details.

\begin{figure}[ht!]
\centering
\includegraphics{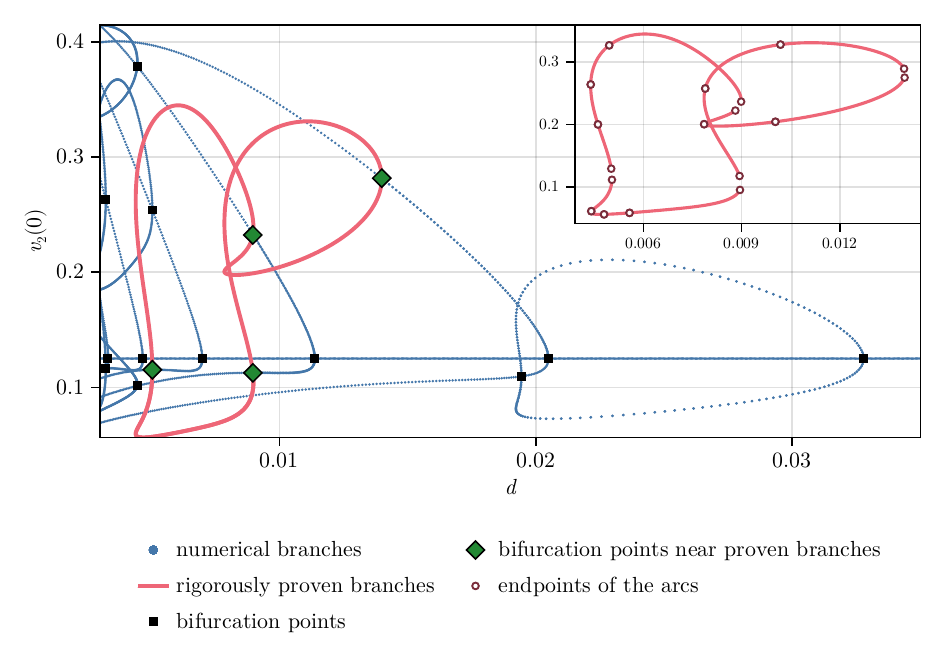}
\caption{Bifurcation diagram of steady-states for the SKT system~\ref{eq:skt} on the domain $(0,1)$, and with $d_{11}=d_{22}=d_{21}=0$, $d_{12}=3$, $r_1=5$, $r_2=2$, $a_1=a_2=3$ and $b_1=b_2=1$.}
\label{fig:skt_bif}
\end{figure}

\paragraph{Results of the code.}
The execution of the code in \cite{CODE} validates the fourteen pieces of curve depicted in red on Figure~\ref{fig:skt_bif}.
The weight sequences are chosen trivial $\chi_k \equiv 1$ and $\omega_n \equiv 1$.
For all pieces, we use $K = 80$ to approximate the steady-states in Fourier and $K_A = 165$ for the finite dimensional approximation of $DF$, while the arclength and the Chebyshev interpolation order $N$ are adapted to each piece.
The computations are reported in Table~\ref{tab:skt_validation}. 

\begin{table}[h!]
\centering
\begin{tabular}{lcccccc}
\toprule
Name & $N$ & Arclength & $Y$ & $Z_1$ & $Z_2$ & $r$ \\
\midrule
u\_bar1a & 16 & 0.38 & $9.34506\times10^{-11}$ & 0.165167 & $1.07247\times10^{4}$ & $1.11941\times10^{-10}$ \\
u\_bar1b & 16 & 0.19 & $5.40238\times10^{-7}$ & 0.267103 & $1.55357\times10^{4}$ & $7.42977\times10^{-7}$ \\
u\_bar1c & 24 & 0.26 & $4.67456\times10^{-9}$ & 0.216850 & $1.60354\times10^{4}$ & $5.96929\times10^{-9}$ \\
u\_bar2a & 16 & 0.40 & $1.23914\times10^{-7}$ & 0.453002 & $2.44590\times10^{4}$ & $2.27692\times10^{-7}$ \\
u\_bar2b & 16 & 0.10 & $2.47429\times10^{-9}$ & 0.628025 & $3.38282\times10^{4}$ & $6.65375\times10^{-9}$ \\
u\_bar2c & 20 & 0.10 & $3.84281\times10^{-7}$ & 0.700900 & $3.86434\times10^{4}$ & $1.41394\times10^{-6}$ \\
u\_bar2d & 24 & 0.32 & $4.40668\times10^{-9}$ & 0.534234 & $3.27280\times10^{4}$ & $9.46429\times10^{-9}$ \\
u\_bar3a & 16 & 0.20 & $6.81293\times10^{-10}$ & 0.402035 & $2.55061\times10^{4}$ & $1.13938\times10^{-9}$ \\
u\_bar3b & 16 & 0.10 & $8.21983\times10^{-9}$ & 0.622886 & $3.48449\times10^{4}$ & $2.18187\times10^{-8}$ \\
u\_bar3c & 16 & 0.10 & $5.59223\times10^{-7}$ & 0.704403 & $3.86498\times10^{4}$ & $2.21161\times10^{-6}$ \\
u\_bar3d & 24 & 0.52 & $1.18497\times10^{-8}$ & 0.574002 & $2.99639\times10^{4}$ & $2.78435\times10^{-8}$ \\
u\_bar4a & 16 & 0.20 & $4.03282\times10^{-9}$ & $0.208423$ & $1.42687\times10^{4}$ & $5.09489\times10^{-9}$ \\
u\_bar4b & 24 & 0.20 & $1.04317\times10^{-7}$ & $0.277004$ & $1.60209\times10^{4}$ & $1.44516\times10^{-7}$ \\
u\_bar4c & 16 & 0.40 & $1.21891\times10^{-10}$ & $0.168895$ & $1.11889\times10^{4}$ & $1.46663\times10^{-10}$ \\
\bottomrule
\end{tabular}
\caption{Validation bounds and resulting error for each of the fourteen pieces depicted on the top right of Figure~\ref{fig:skt_bif}. For each piece, the table reports the Chebyshev truncation order $N$, the chosen arclength, the bounds $Y$, $Z_1$, and $Z_2$, and the validated error bound $r$. The columns ``Name'' refers to the label given to each arc following the convention of the code~\cite{CODE}.}
\label{tab:skt_validation}
\end{table}

Since four bifurcations occur along the main family branch, at least four pieces are required.
Each of these four pieces could, in principle, be validated using a single Chebyshev expansion. However, it is well known that even for smooth functions in practice it can be advantageous to use a piece-wise Chebyshev representation~\cite{AitDri18}.
In our case, the additional subdivisions are motivated by parameter regions where $\| \cM(W) \|_{\B\left(\cS,\cS\right)}$ grows noticeably, making the term $\frac{1}{\pi^2(K_A+1)^2} \| \cM(W) \|_{\B\left(\cS,\cS\right)} \| \cM(DR(\bv)) \|_{\B\left(\cS,\cS\right)}$ dominant in the bound $Z_1$.
Consequently, a relatively large value $K_A = 165$ was necessary to obtain $Z_1 < 1$.
We did not attempt to optimize the subdivision strategy, and a more efficient partition may therefore be possible~\cite{BerShe21}.

\section*{Acknowledgments}

M. Breden and O. H\'{e}not were supported by the ANR project CAPPS: ANR-23-CE40-0004-01.
O. H\'{e}not was also supported by the National Science and Technology Council (NSTC) under grant No. 115-2115-M-002-001-MY2.

\appendix
\section{Bounds for the Cahn--Hilliard equation at fixed parameters}
\label{app:ch}

For the sake of completeness, we briefly derive here the bounds $Y$ and $Z_1$ entering the existence proof of a steady-state of the Cahn--Hilliard equation~\eqref{eq:steadystateCH}.
This appendix is meant as background material for the reader.
A more thorough account can be found, for instance, in \cite[Section~1.3.1]{Bre25}.

In what follows, we fix a parameter value $\lambda = (c, \beta) \in \Lambda$ and write
\[
F(v) = \Delta v + \beta (v - v^{\odot 3} - c),
\]
where we recall that
\[
v = \begin{pmatrix}
v_0 \\ v_1 \\ \vdots
\end{pmatrix} \in \cX, \qquad v(y) = \sum_{k \in \Z} v_{|k|} \cos(\pi k y),
\]
and $\odot$ denotes the product of cosine series in $\cX$.
 
\paragraph{The approximate inverse $A$.}
Since $DF(v) = \Delta + \cM\bigl(\beta(e_0 - 3v^{\odot 2})\bigr)$, the Laplacian dominates in the tail (or, said differently, the multiplication operator is relatively compact with respect to $\Delta$).
We therefore choose as approximate inverse a finite-rank perturbation of $\Delta_0^{-1}$, where $\Delta_0 \bydef \Delta + \Pi_{\le 0}$ so that the inverse is given by
\[
A = A_{\le K} \Pi_{\le K} + \Delta_0^{-1} \Pi_{> K}, \qquad A_{\le K} : \Pi_{\le K} \cX \to \Pi_{\le K} \cX.
\]
The finite block $A_{\le K}$ is obtained by numerically inverting $\Pi_{\le K}DF(\bv)\Pi_{\le K}$.
Because $A$ is a finite-rank perturbation of the inverse Laplacian, $A$ indeed maps $\cY$ into $\cX$ and is a Fredholm operator of index $0$.
 
\paragraph{The bound $Y$.}
The bound $Y$ is straightforward since $\bar{v}$ has finitely many nonzero coefficients, so does $A F(\bar{v})$.
One may simply take $Y = \|A F(\bar{v})\|_{\cX}$, evaluated rigorously with interval arithmetic.
 
\paragraph{The bound $Z_1$.}
The strategy is to split $I - A \, DF(\bar{v})$ column-wise since the $\ell^1$ operator is norm is the supremum of the $\ell^1$ norm of the columns.
The (finite) block of columns of order at most $q$ is handled by a finite computation, while the (infinitely many) remaining columns are estimated analytically, taking advantage of the decay of $\Delta_0^{-1} \Pi_{> K} \to 0$ as $K \to \infty$.
For any splitting order $q$, the operator norm reads
\[
\|I - A \, DF(\bar{v})\|_{\B(\cX,\cX)} = \max \Bigl( \| \Pi_{\le q} - A \, DF(\bar{v}) \Pi_{\le q} \|_{\B(\cX,\cX)}, \, \| \Pi_{> q} - A \, DF(\bar{v}) \Pi_{> q} \|_{\B(\cX,\cX)} \Bigr),
\]
thanks to our choice of a weighted $\ell^1$ norm on $\cX$.
The order $q$ is then chosen so as to make only the factor $\Delta_0^{-1} \Pi_{> K}$ appear in the second term of the maximum (and not the finite part $A_{\leq K}$ of $A$).
Since $\bar{v} \in \Pi_{\le K} \cX$ has finitely many nonzero coefficients, $DF(\bar{v})$ is a banded operator of bandwidth $K$ from which the next identities follow
\[
DF(\bar{v}) \Pi_{\le q} = \Pi_{\le q + K} DF(\bar{v}) \Pi_{\le q}, \qquad
DF(\bar{v}) \Pi_{> q} = \Pi_{> q - K} DF(\bar{v}) \Pi_{> q}.
\]
Hence, a convenient choice is $q = 2K$, and we get
\[
\|I - A \, DF(\bar{v})\|_{\B(\cX,\cX)} = \max \Bigl( \| \Pi_{\le 2K} - \Pi_{\le 3K} A \, DF(\bar{v}) \Pi_{\le 2K} \|_{\B(\cX,\cX)}, \, \| \Pi_{> 2K} - A \Pi_{> K} DF(\bar{v}) \Pi_{> 2K} \|_{\B(\cX,\cX)} \Bigr).
\]
The first term is simply the operator norm of a finite rectangular matrix, which can be computed rigorously with interval arithmetic.
As for the second term, using that $A \Pi_{> K} = \Delta_0^{-1} \Pi_{> K}$ and that $\Delta_0^{-1}$ inverts $\Delta$ on the range of $\Pi_{> 0}$, we obtain
\begin{align*}
\Pi_{> 2K} - A \Pi_{> K} DF(\bar{v}) \Pi_{> 2K}
&= \Pi_{> 2K} - \Delta_0^{-1} \Pi_{> K} DF(\bar{v}) \Pi_{> 2K} \\
&= \Pi_{> 2K} - \Delta_0^{-1} \Pi_{> K} \bigl(\Delta + \cM(\beta(e_0 - 3\bar{v}^{\odot 2})) \bigr) \Pi_{> 2K} \\
&= - \Delta_0^{-1} \Pi_{> K} \, \cM\bigl(\beta(e_0 - 3\bar{v}^{\odot 2})\bigr) \Pi_{> 2K},
\end{align*}
and therefore, since $\|\Delta_0^{-1} \Pi_{> K}\|_{\B(\cX,\cX)} \le \frac{1}{\pi^2 (K+1)^2}$,
\[
\| \Pi_{> 2K} - A \Pi_{> K} DF(\bar{v}) \Pi_{> 2K} \|_{\B(\cX,\cX)} \le \frac{|\beta|}{\pi^2 (K+1)^2} \|e_0 - 3\bar{v}^{\odot 2}\|_{\cX}.
\]
Combining the two estimates, we may take
\[
Z_1 = \max\left( \| \Pi_{\le 2K} - \Pi_{\le 3K} A \, DF(\bar{v}) \Pi_{\le 2K} \|_{\B(\cX,\cX)} , \, \frac{|\beta|}{\pi^2 (K+1)^2} \|e_0 - 3\bar{v}^{\odot 2}\|_{\cX}\right),
\]
and for $K$ sufficiently large we can expect to satisfy the condition $Z_1 < 1$ of Theorem~\ref{th:NK}, as desired.
 
\paragraph{The bound $Z_2$.} Finally, the $Z_2$ estimate follows directly from the Banach algebra property of $\cX$. For all $u\in B_{\rstar}(\bu)$,
\begin{align*}
\|A ( DF(u) - DF(\bu))\|_{\B(\cX,\cX)}
&\le \|A\|_{\B(\cX,\cX)} \|3 \beta (u^{\odot 2} - \bu^{\odot 2})\|_{\cX} \\
&\le 3|\beta|\|A\|_{\B(\cX,\cX)} \|u + \bu\|_{\cX} \|u - \bu\|_{\cX} \\
&\le 3|\beta|\|A\|_{\B(\cX,\cX)} (2\|\bu\|_{\cX} + \rstar) \|u - \bu\|_{\cX},
\end{align*}
hence we take $Z_2 = 3|\beta|\|A\|_{\B(\cX,\cX)} (2\|\bu\|_{\cX} + \rstar)$.


\bibliographystyle{abbrv}
\bibliography{bibfile}

\end{document}